\documentclass[a4paper,10pt]{amsart}
\usepackage{amssymb,amsmath,amsfonts,amsthm}
\usepackage[dvips]{graphicx}
\usepackage{tikz-cd} 
\usepackage{psfrag}
\usepackage{xcolor}
\usepackage{textcomp}

\newtheorem{theorem}{Theorem}[section]

\newtheorem{proposition}[theorem]{Proposition}
\newtheorem{lemma}[theorem]{Lemma}
\newtheorem{corollary}[theorem]{Corollary}

\newtheorem{definition}[theorem]{Definition}
\newtheorem{remark}[theorem]{Remark}

\newtheorem{question}{Question}

\newcommand{\new}[1]{{\bf#1}}

\newcommand{\occult}[1]{}

\usepackage{soul}

\newcommand\Diff{\operatorname{Diff}}

\newcommand\cF{\mathcal F}

\newcommand\Fs{\mathcal F^s}

\newcommand\Fu{\mathcal F^{u}}

\newcommand\supp{\operatorname{supp}}

\begin{document}

\title[Unstable entropy in smooth ergodic theory]{Unstable entropy in smooth ergodic theory
}

\begin{abstract}
In this survey we  recall basic notions of disintegration of measures and entropy along unstable laminations. We  review some  roles of  unstable entropy in smooth ergodic theory including the so-called invariance principle, Margulis construction of measures of maximal entropy, physical measures and rigidity. We also give  some new examples and pose some open problems. 
\end{abstract}

\author{ Ali Tahzibi}

\address{A.~Tahzibi, Instituto de Ci\^{e}ncias Matem\'aticas e de Computa\c{c}\~ao, Universidade de S\~ao Paulo (USP), \emph{E-mail address:} \tt{tahzibi@icmc.usp.br}}

\thanks{A.T is supported by FAPESP 2017/06463-3 and CNPq (PQ) 303025/2015-8. This work was finalized when the author visited Peking University and Université Paris-Sud (financial support: project ANR-16-CE40-0013 ISDEEC).  The conversations with  Sylvain Crovisier, Jiagang Yang, Mauricio Poletti and J\'{e}r\^{o}me Buzzi, were very usefull for improving the presentation and some of the new results and I thank them.
I would like to thank  Yi Shi, Jinhua Zhang and the whole Dynamical system group in Beijing and Suzhou for discussions during the minicourse and seminars. The comments of an anonymous referee were also very usefull for this survey.}

\maketitle

\tableofcontents

\section{Introduction}
The notion of entropy is central in the study of complexity of dynamical systems. It is worthy of attention that entropy, with origins in statistical mechanics and later in information theory, has many roles in the ergodic properties of smooth dynamics and corollaries in several areas of mathematics.
Kolmogorov and Sinai introduced this notion into the field of dynamical systems. 

In this paper we concentrate on some results about the entropy along the expanding direction of a dynamical system. Although ``the whole entropy of a dynamical system is present in its expanding directions", the precise analysis of unstable entropy  is a non trivial task. Sinai (\cite{Sin}) (in uniformly hyperbolic setting), Ledrappier \cite{Ledrappier84} and Ledrappier-Strelcyn \cite{LST}  for the first time constructed measurable partitions and calculated unstable entropy. To define the notion of metric entropy along unstable bundle, one uses the existence of lamination tangent to unstable bundle and  calculate the entropy along such lamination. See Subsection \ref{entropyalong} for general foliations and \ref{3.1} in the partial hyperbolic setting where a simpler way is proposed using \cite{HHW}.

One of the most important references where the notion of unstable entropy is explored deeply  is the celebrated work of Ledrappier and Young (\cite{LY} and \cite{LYii}). In their work, they relate the three fundamental notions of entropy, dimension and Lyapunov exponents of probability invariant measures for smooth dynamics in compact manifolds (Theorem \ref{expdiment}). 

In this survey, all the diffeomorphisms are at least $C^1.$ When we need more regularity, we make it precise. By $C^{1+}$ we mean $\cup_{\alpha > 0} C^{1+\alpha}.$
\\\\
\noindent {\bf Unstable entropy and Physical measures.}
Partially hyperbolic dynamics constitute an important class of systems beyond uniform hyperbolicity. The existence of central bundle  makes the statistical properties of partially hyperbolic dynamics a rich topic of research.  Studying entropy along unstable foliation is a very natural approach to push some techniques applied in uniformly hyperbolic case for the partially hyperbolic diffeomorphisms. 

Pesin and Sinai \cite{PesinSinai} constructed $u-$Gibbs measures for $C^{1+}$ partially hyperbolic attractors and by definition these measures are smooth along unstable leaves. Smoothness of measure along unstable leaves is equivalent to the equality between unstable entropy and sum of the Lyapunov exponents along unstable bundle (See chapter 11 of the book \cite{BDV}). This smoothness is relevant to provide physically meaningful informations on the statistical behaviour of orbits close to the attractor. Indeed, if $f$ is a $C^{1+\alpha}$ dynamics admitting a partially hyperbolic attractor, if a $u-$Gibbs measure is ergodic with negative center Lyapunov exponents, then it is physical in the following sense:
An $f$-invariant probability $\mu$ is physical if its basin has positive Lebesgue measure where the basin of $\mu$ is the following subset:
$$
\{x \in M: \lim_{n \rightarrow \infty} \frac{1}{n} \delta_{f^i(x)} \rightarrow \mu\}.
$$
The proof of the above fact uses absolute continuity of stable laminations and it heavily uses $C^{1+\alpha}, \alpha> 0$ regularity of dynamics.
 The techniques for construction of $u-$Gibbs measures need distortion argument which also uses $C^{1+\alpha}$ regularity.
 
  The notion of unstable entropy is particularly helpful to define appropriately $u-$Gibbs measures for $C^1$ diffeomorphisms (\cite{CYZ}, \cite{HFY}). A $u-$Gibbs measure in $C^1$ setting may be defined as a measure $\mu$ such that unstable entropy coincides with the sum of the Lyapunov exponents along the unstable direction: 
 $$
  h_{\mu}(f, \mathcal{F}^u) = \int \log |\det (Df|E^u) | d\mu.
 $$

  Crovisier-D.Yang-Zhang \cite{CYZ} and T.Hua-F.Yang-J.Yang \cite{HFY} proved that if $\Lambda$ is a partially hyperbolic attractor with decomposition $E^{cs} \oplus E^u$ for a $C^1$ diffeomorphism, then for Lebesgue almost every point $x$ close to attractor, any invariant probability accumulated by empirical measures along the obrit of $x$ is a $u-$Gibbs measure (Theorem \ref{generic1}). As a corollary, if there exists a unique $u-$Gibbs measure then it is physical. 
 \\\\
   \noindent{\bf Unstable entropy and Maximal entropy measures.}
Maximal entropy measures (m.m.e) constitute another natural class of invariant measures.  Bowen \cite{Bowen} and Margulis \cite{Margulis0} obtained measures of maximal entropy for $C^1$ uniformly hyperbolic systems. While Bowen used approximation by periodic measures, Margulis idea was to construct Radon measures  on leaves of invariant foliations which will serve as conditional measures for the measure of maximal entropy. Here we recall the approach introduced by Margulis where  unstable entropy appears in an essential way. We mention that in the uniformly hyperbolic context, the whole entropy of a measure is equal to the unstable entropy and the measure of maximal entropy is obtained locally as product of two Margulis measures defined on leaves of stable and unstable foliations (See Section \ref{s.margulis} for description of Margulis measures.) If $f: M \rightarrow M$ is a topologically transitive Anosov diffeomorphism, there are two families of measures $\{m^u_x\}_{x \in M}, \{m^s_x\}_{x \in M}$ where $m^u_x$ (respectively $m^s_x$) is a Radon measure supported on the unstable leaf (resp. stable leaf) passing through $x$. We have that $m^u_x = m^u_y$ if $y$ and $x$ are in the same unstable leaf. Similarly we have constancy of $m^s_x$ on stable leaves. The family $m^u_x$ (respectively $m^s_x$) is invariant under stable (resp. unstable) holonomy. Up to here, these measures are similar to transverse measures in foliation theory (See Plante \cite{plante75} and Ruelle-Sullivan\cite{RS}.)
However, Margulis measures satisfy an additional dynamical property which makes them right candidates for conditional measures for measures of maximal entropy: the action of $f$ on the family of measures $\{m^u_x\}$ and $\{m^s_x\}$ is given just by a universally constant dilation:
 More precisely, in the Anosov case $$f_* m^u_x = e^{-h_{top}(f)} m^u_{f(x)} \quad \text{and} \quad f_* m^s_x = e^{h_{top}(f)} m^s_{f(x)}.$$ 

For a general partially hyperbolic diffeomorphism $f: M \rightarrow M$ with unstable foliation $\mathcal{F}^u$, Margulis measures  on unstable leaves can be defined as measurable family $\{m^u_x\}_{x \in M}$ of non-trivial Radon measures on unstable leaves where the dynamics acts by a dilation: here exists a dilation factor $D > 0$ such that $f_* m_x = D^{-1} m_{f(x)}.$ See Definition \ref{def.margulisfamily}.

These family of measures can be  used to find invariant measures maximizing entropy along unstable direction or even m.m.e's in some cases.   Buzzi-Tahzibi-Fisher \cite{BFT} constructed Margulis measures for a large class of partially hyperbolic dynamics called flow type diffeomorphisms and gave a dichotomy on measures of maximal entropy for this class.  In their paper, a general convexity result (Proposition \ref{prop-Ledrappier}) shows that if $f$ is a $C^2$ partially hyperbolic diffeomorphism with a measurable Margulis family of measures $\{m^u_x\}_{x \in M}$   with dilation factor $D_u,$ i.e $f_* m^u_x = D_u^{-1} m^u_{f(x)}$ and $m^u_x$ is fully supported on the unstable leaf passing through $x$ then
$$
 h_{\mu}(f, \mathcal{F}^u) \leq \log (D_u).
$$
Moreover, the equality holds if and only if the disintegration of $\mu$ along $\mathcal{F}^u$ coincides with $m^u.$ This is a key argument to prove uniqueness of hyperbolic measure of maximal entropy (one m.m.e with positive center exponent and one with negative exponent) for flow type diffeomorphisms in \cite{BFT}.
 
 We mention that the Carath\'{e}odory constructions developped in (\cite{CPZ1} and \cite{CPZ2}, 
\cite{H89}) 
also provide Margulis measures.
\\\\
\noindent{\bf Unstable entropy and Invariance Principle.}
In this paper we also recall  a celebrated theorem of Furstenberg \cite{Furstenberg} and its generalization by Ledrappier \cite{L} about Lyapunov exponents of random product of matrices and linear cocycles. Avila and Viana \cite{AV} made an essential contribution in this theory by generalizing Ledrappier's ideas to non-linear cocycles and they called it invariance principle. 

They considered smooth cocycles, that is, fiber bundle morphisms acting by diffeomorphisms on the fibers. 
 Linear cocycles
can be studied through their induced action on the projective bundle associated to the
underlying vector bundle. 

The extremal Lyapunov exponents measure the smallest and largest exponential rates of growth of the derivative along the fibers. The Invariance Principle states that if these two numbers coincide then ``the fibers carry some amount of structure which is transversely invariant, that is, invariant under certain canonical families of homeomorphisms between fibers." 

Following the works of Avila-Viana \cite{AV} and Avila-Santamaria-Viana-Wilkinson \cite{ASVW}  we divide invariance principle results into two parts:

 First, let $f: M \rightarrow M$ be an Anosov or partially hyperbolic diffeomorphism and $F: \mathcal{E} \rightarrow \mathcal{E}$ be fibered dynamics over $f$ and $\pi: \mathcal{E} \rightarrow M$ the fiber projection.  Under suitable assumptions, the invariant (stable and unstable) foliations
of the base dynamics (uniformly hyperbolic or partially hyperbolic) lift to invariant foliations of the fibered system which we call by stable and unstable foliations of $F$.

Avila-Viana \cite{AV} generalized Ledrappier result and proved that if $\mu$ is an $F-$invariant measure with non-positive Lyapunov exponents along fibers then $\mu$ disintegrated into $\{\mu_x\}$ defined almost everywhere and $\mu_x$ is invariant under unstable holonomy.  This can be encapsulated into the following: 
{\it ``non-positive Lyapunov exponents yields determinism."} Clearly if all Lyapunov exponents are non-negative we obtain stable holonomy invariance almost everywhere. 

The second part of the invariance principle results is technically similar to classical Hopf argument (in the uniformly hyperbolic case) and Julienne density argument (in partially hyperbolic context \cite{BW}, \cite{GPS}) in the  proof of ergodicity. 

Now we assume that all Lyapunov exponents along fiber vanish. So by the first part we have both 
invariance under stable and unstable holonomies.  

If the base dynamics $f$ is uniformly hyperbolic and $\pi_* \mu$ has local product structure or $f$ is partially hyperbolic, volume preserving and center bunched  any measurable section which is essentially (i.e., almost everywhere) saturated under
the lifted stable foliation and essentially saturated under the lifted unstable foliation
coincides, almost everywhere, with some section that is saturated by both lifted foliations.

 Moreover, if the base diffeomorphism is accessible then such a bi-saturated
section may be chosen to be continuous. So, in this second part of invariance principle result, the disintegration of $\mu$ is defined {\it everywhere} and continuous.

Invariance principle had been used extensively in the analysis of measures of maximal entropy, physical measure and continuity of Lyapunov exponents of cocycles. See for some examples the results in
\cite{UVY}, \cite{AV}, \cite{Karina}, \cite{OP}, \cite{BBB}.

In Subsection \ref{general} we describe a  result of Tahzibi-Yang \cite{TY} (written slightly more general than in the original paper) which shed light on (first part of) nonlinear invariance principle using unstable entropy. Our result provides an invariance principle formulated in terms of unstable entropy. So, in particular we may get the previous results (first part of invariance principle) which are formulated in terms of Lyapunov exponents. 

In particular we give some applications of this invariance principle to understand measures of maximal entropy and hyperbolicity of high entropy measures for partially hyperbolic diffeomorphisms with one dimensional compact center leaves. In general it seems interesting to understand the hyperbolicity of high entropy measures of a partially hyperbolic dynamics whenever all ergodic m.m.e's are hyperbolic. We also propose to undestand a variational principle type result for non-hyperbolic measures (Question \ref{q.vp}). See \cite{DGR} for similar results in the iterated function systems setting. 

In Subsection \ref{bochi}  we show briefly how the ideas from this invariance principle can be used to prove that Bochi-Ma\~{n}\'{e} result is not correct in non-invertible setting.

The above mentioned invariance principle results require a global fiber bundle structure. 
A large class of partially hyperbolic diffeomorphism with compact center leaves can be seen as  fibered maps. This makes invariance principle a powerfull tool to study smooth ergodic theory of fibered partially hyperbolic systems.

Recently, S. Crovisier and M. Poletti announced an invariance principle result which deals with partially hyperbolic diffeomorphisms which act quasi-isometrically on center leaves.  Similar to \cite{TY}, this invariance principle is also formulated in terms of unstable entropy. Their result in particular has the advantage of being adapted to work with non-smooth measures and partially hyperbolic dynamics without compact center leaves. 
\\\\
\noindent{ \bf Unstable entropy and Rigidity.}
We mention that the idea of considering entropy along unstable foliation (or more generally expanding foliations) is used in the proof of rigidity results too.  If $f$ is a $C^1$ diffeomorphism and $\mathcal{F}$ is an invariant expanding foliation, i.e. there exists $\lambda > 1$ such that $\forall x \in M, \|Df^{-1}|T_x(\mathcal{F})\| \leq \lambda^{-1},$ then the notion of entropy along $\mathcal{F}$ is well defined and we call an invariant measure $\mu$ as Gibbs expanding state along $\mathcal{F}$ if it satisfies ``Pesin's entropy formula": 
$$
 h_{\mu}(f, \mathcal{F}) = \int \log |det (Df | T_x\mathcal{F})| d\mu(x).
$$
If $f$ is $C^r, r > 1$ then this definition is equivalent to say that $\mu$ has absolutely continuous disintegration along leaves of $\mathcal{F}$ (See Theorem \ref{SaYtheorem}).

A topological conjugacy between two Anosov diffeomorphism is in general H\"{o}lder continuous. If the conjugacy is $C^1$, then smooth invariants, for instance  the periodic data  (eigenvalues of the return map at the periodic points), must be preserved by the conjugacy. While coincidence of periodic data for conjugate Anosov diffeomorphisms of $\mathbb{T}^2$ implies smooth conjugacy (de La Llave and M. Moriy\'{o}n \cite{LlM}), the same kind of conclusion is not correct in higher dimensions. There is a vast number of references about smooth conjugacy and we refer the reader to the introduction of the paper by R. Saghin and J. Yang \cite{SaY} to see many of them.  In three dimensional case, See \cite{Gogu}, \cite{SaY}, \cite{MiTa} for some of the results in Rigidity of  Anosov diffeomorphisms on $\mathbb{T}^3$.

 A core idea in the proof of rigidity results for Anosov diffeomorphisms is the following (See Theorem \ref{SaYtheorem2}): Suppose that  
$f, g \in \Diff^k(M), k \geq 2$ and $\mathcal{F}^f, \mathcal{F}^g$ expanding foliations invariant by $f$ and $g$ with uniform $C^r, r > 1$ leaves. An $f-$invariant measure is called Gibbs state along $\mathcal{F}^f$ if its disintegration along $\mathcal{F}^f$ is absolutely continuous with respect to Lebesgue measure on leaves of $\mathcal{F}^f.$

Suppose $f$ and $g$ are conjugate by a homeomorphism $h$ which maps $\mathcal{F}^f$ to $\mathcal{F}^g$. If $\mu$ is Gibbs expanding state with respect to $\mathcal{F}^f$ then $h_{*}\mu$ is Gibbs expanding state along $\mathcal{F}^g$  if and only if the sum of Lyapunov exponents of $\mu$ along $\mathcal{F}^f$ coincides with the sum of exponents of $\nu$ along $\mathcal{F}^g:$
$$
\int \log (\det (Df|T\mathcal{F}^{f})) d\mu =  \int \log (\det (Dg|T\mathcal{F}^{g})) d\nu.
$$
If $\mathcal{F}^f$ and $\mathcal{F}^g$ are one dimensional then the above if and only if statement is equivalent to: $h$ restricted on each $\mathcal{F}^f$ leaf within the support of $\mu$ is uniformly $C^r.$
In theorems \ref{SaY}, \ref{MiTa} a main step of the proof is to prove differentiability of conjugacy between an Anosov diffeomorphism and its linearization along invariant one dimensional foliations and the  one dimensional argument is crucial.

\section*{ Structure of the paper}

In Section \ref{measurabletoolbox} we provide elementary results in measure disintegration. In Subsection \ref{sub.disintegration} we recall the celebrated Rokhlin disintegraiton theorem and we prove an easy but fundamental lemma (Lemma \ref{product}) which will be used in the following sections dealing with holonomy invariance of conditional measures. 
In Subsection \ref{entropyalong} the metric entropy along foliations and measurable subordinated partitions are discussed.

In Section \ref{ph} we study the role of unstable entropy in understanding better partially hyperbolic dynamics.  
Firstly in Subsections \ref{3.1} and \ref{3.2} we review variational principle result, relation between unstable entropy and Lyapunov exponents. In particular Pesin formula and Margulis-Ruelle inequality for unstable entropy  are discussed.
 We 
mention the relation between maximal entropy measures and those which maximize unstable entropy ($u-$maximal measures) and propose some directions to follow to obtain new results. See Question \ref{rigidity} and comments after it.

 We  postpone the construction of an ergodic measure of maximal entropy which is not $u-$maximal to the Section \ref{invprincple}. Indeed, our argument envolved in such example uses invariance principle which is discussed in Section \ref{invprincple}.
In Subsection \ref{3.3} we study $u-$Gibbs measures and their relation with empirical measures.  The results in this subsection are in particular fundamental to understand statistical properties of partially hyperbolic attractors with just $C^1$ regularity.  
Finally in Subsection \ref{3.4} we review an argument by J. Yang to prove topological transitivity using techniques from unstable entropy.

In Section \ref{invprincple} we begin with the model of random product of matrices, introduced by Furstenberg and prove Furstenberg's result using non linear version of invariance principle proved by Avila-Viana. In Subsection \ref{general} we introduce a general setting of non linear cocycles and prove an invariance principle. This invariance principle is formulated in terms of unstable entropy.

In Subsection \ref{applicationinv} we give some applications of the invariance principle which is formulated in terms of entropy. More precisely, in \ref{mmenotu} we proved example of a partially hyperbolic diffeomorphism with an ergodic measure of maximal entropy which does not maximize unstable entropy. In \ref{highentropy} we recall a dichotomy theorem proved by Hertz-Hertz-Tahzibi-Ures (Theorem \ref{dichotomy})  and in the setting of their theorem we prove that for generic diffeomorphisms high entropy measures are hyperbolic. Finally in \ref{bochi} we see that the so called Bochi-Ma\~{n}\'{e} result in non-invertible setting does not hold.

In Subsection \ref{cropol} we announce a recent result of Crovisier-Poletti on an invariance principle for partially hyperbolic dynamics without compact center leaves which act quasi-isometrically on center leaves.

In Section \ref{s.margulis} we introduce the notion of Margulis measures and their properties related to maximal measures using convexity arguments. In \ref{margulishyperbolic} we see how to see Margulis family of measures as conditional measures of maximal entropy measures in uniformly hyperbolic case. Finally in \ref{margulisph} we extend some of the results to partially hyperbolic setting and pose some questions and directions for further work.

In Section \ref{s.rigidity} we apply entropy arguments to some rigidity problems. In Subsection \ref{rigidity-gibbs} we recall a rigidity result envolving invariant expanding foliations and we apply it in \ref{rigidityanosov} to obtain rigidity of Anosov diffeomorphisms in $\mathbb{T}^3.$
Finally in \ref{measurerigidity} we give a very sketchy proof of a simple example (containing main ideas) of a  Katok-Spatzier theorem on measure rigidity of actions.

\section{Preliminaries} \label{measurabletoolbox}

\subsection{Disintegration of measures} \label{sub.disintegration}

In this section we develop an abstract measurable toolbox which deals with disintegration of measures and their properties.

Let $(M, \mathcal{B}, \mu)$ be a probability measure space and $\mathcal{P} = \{P_i\}_{i \in \mathcal{I}}$ a partition of $M$ into measurable subsets, i.e $M = \bigcup_{P_i \in \mathcal{P}} P_i$ such that $P_i \in \mathcal{B},  \mu(P_i \cap P_j) = 0, \mu(\bigcup P_i) = \mu(M) =1.$ We denote by $\sigma(\mathcal{P})$ the smallest $\sigma-$algebra containing all $P_i \in \mathcal{P}.$

Given two partitions $\mathcal{P}$ and $\mathcal{Q}$ we use the usual notation $\mathcal{P} \leq \mathcal{Q}$ if and only if $\mathcal{B}(\mathcal{P}) \subseteq \mathcal{B}(\mathcal{Q})$ or in other words any $P_i = \bigcup_{j \in \mathcal{I}_i} Q_j$. 

In what follows we recall the definition of {\it measurable partitions}. We emphasize that a partition into measurable subsets is not necessarily a measurable partition. Although the notation is a bit confusing, it is folkloric in this area to reserve the name of measurable partition  for a ``countably generated" partition into measurable subsets, as following:
\begin{definition} \label{def:mensuravel}
We say that a partition $\mathcal P$ is measurable (or countably generated) with respect to $\mu$ if there exist a measurable family $\{A_i\}_{i \in \mathbb N}$ and a measurable set $F$ of full measure such that
if $B \in \mathcal P$, then there exists a sequence $\{B_i\}$, where $B_i \in \{A_i, A_i^c \}$ such that $B \cap F = \bigcap_i B_i \cap F$.
\end{definition}

There is another equivalent way to define measurable partitions (See chapter 5 of Einsiedler and Ward's book \cite{EW}.) where firstly we correspond to a $\sigma-$algebra $\mathcal{A} \subseteq \mathcal{B}$ a partition $\mathbb{P}(\mathcal{A})$ and then a partition $\mathcal{P}$ into measurable subsets is called measurable iff 
$$
\mathbb{P}(\sigma (\mathcal{P})) = \mathcal{P}.
$$
where $\sigma(\mathcal{P})$ is the $\sigma-$algebra generated by the partition $\mathcal{P}.$

Example: Partition into the orbits of irrational flow in $\mathbb{T}^2$ is not measurable. Here $\mathbb{P}(\sigma (\mathcal{P}))$ is the trivial partition of the whole space as one atom. Indeed, $\sigma(\mathcal{P})$ is the trivial $\sigma-$algebra $\{\emptyset, \mathbb{T}^2\}.$

 Let $\mathcal P$ be a measurable partition of a compact metric space $M$ and $\mu$ a borelian probability. Let $\pi : M \rightarrow M / \mathcal{P}$ the canonical projection from $M$ to quotient space defined by the atoms of the partition. There is a natural measure structure on $M/\mathcal{P}$ defined by the measure $\tilde{\mu} = \pi_* \mu$ on the corresponding pushed $\sigma-$algebra. We denote by $\mathcal{P}(x)$ the element of partition which contains $x.$

 \begin{definition} \label{definition:conditionalmeasure}
 Given a partition $\mathcal P$, a family $\{\mu_P\}_{P \in \mathcal P} $ is a \textit{system of conditional measures} for $\mu$ (with respect to $\mathcal P$) if for given $\phi \in C^0(M):$
\begin{itemize}
 \item[i)] $P \mapsto \int \phi \mu_P$ is $\widetilde \mu$- measurable;
\item[ii)] $\mu_P(P)=1$ $\widetilde \mu$-a.e.;
\item[iii)]
$\displaystyle{ \int_M \phi d\mu = \int_{M / \mathcal P}\left(\int_P \phi d\mu_P \right)d\widetilde \mu } = \int_M \int_{\mathcal{P}(x)} \phi d\mu_{\mathcal{P}(x)} d\mu.$
\end{itemize}
\end{definition}
  By Rokhlin's theorem \cite{Rok49},  there exists a unique disintegration by conditional probabilities for $\mu$. Clearly, uniqueness should be understood almost everywhere.
 \begin{theorem}[Rokhlin's disintegration \cite{Ro52}] \label{teo:rokhlin} 
 Let $\mathcal P$ be a measurable partition of a compact metric space $M$ and $\mu$ a borelian probability. Then there exists a unique disintegration by conditional probability measures for $\mu$.
\end{theorem}

\subsubsection{Product structure}
Let $(A:=[0, 1]^{m+n}, \mu)$ be the unit square equipped with a probability measure and $\mathcal{F} , \mathcal{G}$ be a pair of  transverse foliations of $A$ with compact leaves and dimension of leaves respectively $m$ and $n$. 

We assume the following topological  product structure: For some $x_0 \in A,$ there exists a continuous injective and surjective map $Q(. , .) : \mathcal{F}(x_0) \times \mathcal{G}(x_0) \rightarrow [0, 1]^{m+n}$ such that $Q(x, y) = \mathcal{G}(x) \cap \mathcal{F}(y).$

 The foliations $\mathcal{F}, \mathcal{G}$ can be considered as measurable partitions. This can be seen as we are dealing with a product (of two separable metric spaces) structure.

We denote by $\{\mu_{x}^{\mathcal{G}}\}$ and $\{\mu^{\mathcal{F}}_{x}\}$ the system of conditional probability measures along  $\mathcal{G}$ and $\mathcal{F}:$
$$\mu= \int_{A} \mu^{\mathcal{F}}_{x} d\mu (x) = \int_{A} \mu^{\mathcal{G}}_x d \mu(x)$$
where $\mu_ {x}^{\mathcal{F}}$ (resp. $\mu^{\mathcal{G}}_x$) is a probability measure supported on the leaf $\mathcal{F} (x)$ (resp. $\mathcal{G}(x)$).

Another equivalent way to write the disintegration equation (along $\mathcal{G}$) above  is to consider the quotient space $A/ \mathcal{G}$ equipped with the quotient measure $\tilde{\mu}:= \pi_*(\mu)$ where $\pi: A \rightarrow A/\mathcal{G}$ is the canonical projection. Clearly $A /\mathcal{G}$ can be identified with $[0,1]^m.$ We can write
$$
 \mu =  \int_{A/\mathcal{G}} \mu_{P}^{\mathcal{G}} d \tilde{\mu} (P)
$$
where $\mu^{\mathcal{G}}_{P}$ is the conditional probability measure on $P$, a leaf of $\mathcal{G}$. By definition,  for any integrable function $\phi: M\to \mathbb{R}$, we have
$$\int_{A} \phi d \mu = \int_{A/\mathcal{G}} \int_{P} \phi(x) d \mu_{P}^{\mathcal{G}} (x) d \tilde{\mu}(P).  $$

The product structure of the pair of foliations above, permits us to define holonomy maps $H^{\mathcal{F}}$ and $H^{\mathcal{G}}$ respectively between leaves of $\mathcal{G}$ and $\mathcal{F}.$

We say a system of disintegration $\mu^{\mathcal{G}}$ is invariant by holonomy of $\mathcal{F}$, or \emph{$\mathcal{F}-$invariant} if $$\mu_y^{\mathcal{G}} = (H^{\mathcal{F}}_{x, y})_{*} \mu_x^{\mathcal{G}}$$ for
 $x,y$ belonging to a $\mu-$ full measure subset, where $H_{x, y}^{\mathcal{F}}$ is the $\mathcal{F}-$holonomy map between $\mathcal{G}(x)$ and $\mathcal{G}(y)$ induced by the foliation $\mathcal{F}$. Similarly we define \emph{$\mathcal{G}-$invariance} of $\{\mu_{y}^{\mathcal{F}}\}$ by $\mu_{y}^{\mathcal{F}} = (H^{\mathcal{G}}_{x, y})_{*} \mu^{\mathcal{F}}_{x}.$
\begin{lemma} \label{product}
If $\{\mu_{x}^{\mathcal{G}}\}$ is $\mathcal{F}-$invariant then  $\{\mu_x^{\mathcal{F}}\}$ is $\mathcal{G}-$invariant and $\mu = Q_* (\mu^{\mathcal{F}}_{x_0} \times \mu^{\mathcal{G}}_{x_0})$ for $\mu-$almost every $x_0.$  Similarly, if $\{\mu^{\mathcal{F}}_x\}$ is $\mathcal{G}-$invariant  then $\{\mu^{\mathcal{G}}_x\}$ is $\mathcal{F}-$invariant.
\end{lemma}
\begin{proof}
We prove that if $\mu^{\mathcal{G}}$ is $\mathcal{F}-$invariant then $\mu^{\mathcal{F}}$ is $\mathcal{G}-$invariant.  
By the definition of conditional measures and $\mathcal{F}-$invariance of $\mu^{\mathcal{G}}$ we have
$$ \int \phi d \mu = \int_{\mathcal{F}(x_0)} \int_{\mathcal{G}(x_0)} \phi \circ H^{\mathcal{F}}_{x_0, z}(x) d \mu^{\mathcal{G}}_{x_0}(x) d \tilde{\mu} (z)
$$
for any continuous function $\phi$ where $\tilde{\mu}$ is the quotient measure on the quotient space identified with $\mathcal{F}(x_0).$
Using Fubini's theorem and the fact that $H^{\mathcal{F}}_{x_0, z}(x) = H^{\mathcal{G}}_{x_0, x} (z)$ for any $x \in \mathcal{G}(x_0) , z \in \mathcal{F}(x_0)$ we obtain
$$
\int \phi d \mu = \int_{\mathcal{G}(x_0)} \int_{\mathcal{F}(x_0)} \phi \circ H^{\mathcal{G}}_{x_0, x}(z) d \tilde{\mu}(z) d \mu^{\mathcal{G}}_{x_0}(x).
$$
We see that by the above equation, the quotient measure (quotien by $\mathcal{F}$ leaves) can be identified by $\mu^{\mathcal{G}}_{x_0}$ and
by essential uniqueness of disintegration,  the system of conditional measures $\{ \mu^{\mathcal{F}}_x\}$ satisfies $\mu^{\mathcal{F}}_{x} = (H^{\mathcal{G}}_{x_0, x})_{*} \tilde{\mu}$. This implies $\mathcal{G}-$invariance of conditional measures along $\mathcal{F}$.
The last claim of the lemma is direct from definition and Fubini theorem.
\end{proof}

In general for two transverse foliations $\mathcal{F}_1, \mathcal{F}_2$ defined on a compact manifold $M$ and $\mu$ a probability measure on $M$ we say that $\mu$ has local product structure with respect to $\mathcal{F}_1$ and $\mathcal{F}_2$ if for almost every $x \in M$ there is a local chart $B_x$ which trivializes both foliations and the normalized restriction of $\mu$ on $B_x$  satisfies invariance by holonomies of both of foliations in $B_x$ as in the above lemma. 
Sometimes, it is convenient to consider a relaxed form of local product structure for measures. We prefer to call this weaker definition by quasi-product structure.  In \cite{BFT} we called $\mu$ has quasi-product structure if $\mu$ is equivalent to a product measure. In the literature, many authors call this relaxed version by local product structure. However, as there are some subtle differences between these two definitions, we prefer to separate them.

\subsection{Metric Entropy along foliations} \label{entropyalong}
Let $f: M \rightarrow M$ be a diffeomorphism of a compact manifold preserving a probability measure $\mu.$ Let $\mathcal{P} = \{P_1, P_2, \cdots P_n\}$ be a partition into measurable subsets. Then 
\begin{equation}
H_{\mu}(\mathcal{P})= \sum_{i=1}^{n} - \mu(P_i) \log(\mu(P_i)) = \int_M - \log(\mu(\mathcal{P}(x))) d \mu 
\end{equation}
  where $\mathcal{P}(x) = P_i$ for $x \in P_i$ and, 
\begin{equation} \label{defentropy}
h_{\mu}(\mathcal{P}, f) = \lim_{n \rightarrow \infty} \frac{1}{n} H_{\mu}(\mathcal{P}^n) = \lim_{n \rightarrow \infty} H_{\mu}(\mathcal{P}^n)-H_{\mu}(\mathcal{P}^{n-1})
\end{equation}
where $\mathcal{P}^n = \mathcal{P} \vee f^{-1}(\mathcal{P}) \vee \cdots \vee f^{-n+1}(\mathcal{P})$ is a refinement of $\mathcal{P}$ by the action of $f$.

 Let $\mathcal{F}$ be an  invariant foliation (lamination), i.e, $f(\mathcal{F}(x)) = \mathcal{F}(f(x)).$  One way to define the entropy of an invariant measure along an invariant foliation is to use measurable partitions subordinated to $\mathcal{F}$:
 
 A measurable partition (see definition \ref{def:mensuravel}) $\xi$ is called increasing and subordinated to $\mathcal{F}$ if it satisfies:
\begin{itemize}
\item (a) $\xi(x) \subseteq \mathcal{F}(x)$ for $\mu-$almost every $x,$
\item (b)$f^{-1}(\xi) \geq \xi$ (increasing property),
\item (c) $\xi(x)$ contains an open neighbourhood of $x$ in $\mathcal{F}( x)$ for $\mu-$almost every $x$.     
\end{itemize}

The existence of a measurable increasing partition subordinated to an invariant  lamination in general is a delicate problem. It requires a Markov type property (item (b)). For a uniformly expanding foliation invariant by a diffeomorphism, we always have such a partition (See \cite{Yang}.) We say an invariant foliation $\mathcal{F}$ is uniformly expanding if there exists $ \lambda > 1$ such that $\forall x \in M, \|Df^{-1}|T_x(\mathcal{F}(x))\| < \lambda^{-1}.$
Observe that if $\mathcal{F}$ is an expanding foliation, we have more usefull properties for a partition $\xi$ which is increasing and subordinated:
\begin{itemize} 
 \item (d) $\bigvee_{n=0}^{\infty} f^{-n}\xi $ is the partition into points;
  \item (e) the largest $\sigma-$algebra contained in $\bigcap_{n=0}^{\infty} f^n(\xi)$ is $\mathcal{B}_{\mathcal{F}}$
where $\mathcal{B}_{\mathcal{F}}$ is the $\sigma-$algebra of $\mathcal{F}-$saturated measurable subsets (union of entire leaves).
\end{itemize}

Given a measure $\mu$ and two measurable partitions $\alpha, \beta$ of $M$, let 
$$ H_{\mu}(\alpha | \beta) := \int_{M} \log (\mu^{\beta}_x(\alpha(x))) d\mu(x)$$ denote the conditional entropy of $\alpha$ given $\beta.$ Here $\mu^{\beta}_x$ is the conditional measure of $\mu$ with respect to $\beta$ (See Section \ref{measurabletoolbox} for a preliminary on conditional measures.) 

For any invariant lamination $\mathcal{F}$ with a measurable subordinated increasing partition satisfying all the properties including (d) and (e) above, we define 
$$h_{\mu}(f, \mathcal{F}) := h_{\mu}(f, \xi) = H_{\mu} (f^{-1} \xi | \xi) = H_{\mu}(\xi | f \xi).$$

It can be proved that the above definition is independent of $\xi$. Indeed let $\xi_1$ and $\xi_2$ be two such measurable partitions then for any $n \in \mathbb{N}$:
$$h_{\mu}(f, \xi_1 \vee \xi_2) = h_{\mu}(f, \xi_1 \vee f^n(\xi_2)).$$
The above equation comes from the definition of the entropy (\ref{defentropy}) and the fact that $\xi_2$ is increasing.
 Now, 

\begin{equation}
\begin{split}
h_{\mu}(f, \xi_1 \vee f^n\xi_2) =& H_{\mu}( f\xi_1 \vee f^{n+1}\xi_2 |\xi_1 \vee f^{n}\xi_2 ) \\
= & H_{\mu}(f\xi_1 | \xi_1 \vee f^{n}\xi_2 ) + H_{\mu}(f^{n+1}\xi_2 | f\xi_1 \vee \xi_1 \vee f^{n}\xi_2  ) \\
= &   H_{\mu}(f\xi_1 | \xi_1 \vee f^{n}\xi_2 ) + H_{\mu}(f^{n+1}\xi_2 |  \xi_1 \vee f^{n}\xi_2  )\\
=&H_{\mu}(f\xi_1 | \xi_1 \vee f^{n}\xi_2 ) + H_{\mu}(f\xi_2 | f^{-n} \xi_1 \vee \xi_2  )
\end{split}
\end{equation}
As $n \rightarrow \infty$ ,  $H_{\mu}(f\xi_1 | \xi_1 \vee f^{n}\xi_2 ) \rightarrow H_{\mu}(f\xi_1 | \xi_1) = h_{\mu}(f, \xi_1)$ and $H_{\mu}(f\xi_2 | f^{-n} \xi_1 \vee \xi_2  ) \rightarrow 0.$
The first convergence can be proved using property (e) and the second one is a corollary of the fact that $\xi_1$ is generating (property (d)). See Lemma 3.1.2 in \cite{LY} for more details. By a similar argument changing $\xi_1$ and $\xi_2$  we conclude that 
$$
h_{\mu}(f, \xi_2) = h_{\mu}(f, \xi_1 \vee \xi_2) = h_{\mu}(f, \xi_1). 
$$

\section{Partial hyperbolicity and non-uniformly hyperbolic setting} \label{ph}
In a more general  than uniformly expanding setting,
Ledrappier and Young \cite{LYii} defined a hierarchy of unstable entropies for smooth systems with unstable lamination (in the sense of Pesin). Let $f$ be a $C^{1+\alpha}$ diffeomorphism of a compact manifold and $\mu$ be any ergodic  invariant probability measure with $u$ positive Lyapunov exponents $\lambda_1 > \lambda_2 > \cdots > \lambda_u > 0$. There exists a filtration of laminations $\mathcal{F}_i$ where $\mathcal{F}_1(x) \subset \mathcal
{F}_2(x) \cdots \subset \mathcal{F}_{u}(x)$ and for $y \in \mathcal{F}_i(x):$
$$
 \limsup_{n \rightarrow -\infty} \frac{1}{n} \log d(f^n(x), f^n(y)) \leq -\lambda_i.
$$ 
The above laminations are tangent to Oseledets' filtration of tangent space defined  $\mu-$almost everywhere.
They constructed a nested family of partitions $\xi_1 \geq \xi_2 \geq \cdots \geq \xi_u$ where each $\xi_i$ is generator and increasingly subordinated to $\mathcal{F}_i.$
By means of these partitions it is possible to define an increasing sequence of unstable entropies $h_i$ corresponding to the nested family of laminations $\mathcal{F}_i:$
$$
 h_i = h_{\mu} (f, \mathcal{F}_i) := H_{\mu}(f^{-1}\xi_i | \xi_i).
$$ 
It is also proved that the above entropies do not depend on the choice of $\xi_i$.

Ledrappier and Young proved:

\begin{theorem} {\normalfont (\cite{LYii}) } \label{expdiment}
Let $f: M \rightarrow M$ be a $C^2$ diffeomorphis of a compact manifold and $\mu$ be an ergodic Borel probability invariant measure. Let $\lambda_1 > \cdots > \lambda_u$ denote the distinct positive Lyapunov exponents of $f$,  and let $\delta_i$ be the dimension of $\mu$ on $\mathcal{F}_i.$ Then for $1 \leq i \leq u$ there are numbers $\gamma_i$ with $0 \leq \gamma_i \leq \dim(E_i)$ such that $\delta_i = \sum_{j \leq i} \gamma_i$ and
\begin{itemize}
\item $h_1 = \lambda_1 \delta_1$
\item $h_i  - h_{i-1} = \lambda_i (\delta_i - \delta_{i-1})$ for $i=2, \cdots, u,$ and
\item $h_u = h_{\mu}(f).$
\end{itemize}
\end{theorem}

By a result of Ledrappier and Xie \cite{LX}: in the setting of the above theorem if the transverse entropy vanishes, i.e, $h_i = h_{i-1}$ then for $\mu-$almost every $x$ the conditional measure along $\mathcal{F}_i (x)$ is carried by a single leaf of $\mathcal{F}_{i-1}.$

\begin{remark}
The above theorem of Ledrappier and Young needs $C^2$ regularity. However, Aaron Brown \cite{brownl} announced a proof using just $C^{1+\alpha}$ regularity.
\end{remark}

\subsection{Partially hyperbolic setting } \label{3.1}

$f: M \rightarrow M$ is \emph{partially hyperbolic} if
there is a $Df$-invariant splitting of the tangent
bundle $TM = E^s\oplus E^c \oplus E^u$ such that for all unit vectors
$v^{*}\in E^{*}_x\setminus \{0\}$ ($*= s, c, u$) with $x\in M$ we have:
$$\|D_xfv^s\| < \|D_xfv^c\| < \|D_xfv^u\| $$
for some suitable Riemannian metric.  Furthermore $f$ satisfies:
$\|Df|_{E^s}\| < 1$ and $\|Df^{-1}|_{E^u}\| < 1$.  For partially hyperbolic diffeomorphisms, it is a well-known fact that there are invariant foliations $\cF^{*}$ tangent to the distributions $E^{*}$
for $*=s,u$ . The leaf of $\cF^{*}$
containing $x$ will be called $\cF^{*}(x)$, for $*=s,u$. \par In general the central bundle $E^c$ may not be tangent to an invariant foliation (See \cite{RHU}). However, whenever it exists we denote it by $\mathcal{F}^c.$
 
 One says that $f$ is {\it accessible} if every $x, y \in M$ can be joint by a path which is union of finitely many $C^1$-curves inside $\mathcal{F}^s$ or $\mathcal{F}^u.$ We say $f$ is {\it dynamically coherent} if there exist invariant foliations $\mathcal{F}^{cs}$ and $\mathcal{F}^{cu}$ tangent to the bundles $E^s \oplus E^c$ and $E^u \oplus E^c.$ The intersection of these two foliations define a center foliation $\mathcal{F}^c.$ For every $x \in M, \mathcal{F}^c(x):= \mathcal{F}^{cs}(x) \cap \mathcal{F}^{cu}(x).$
 
 For a general dynamically coherent partially hyperbolic diffeomorphism, if $x, y$ are in the same unstable leaf, then there is a well  defined local holonomy between neighbourhoods of $x$ and $y$ in $\mathcal{F}^c(x)$ and $\mathcal{F}^c(y).$ However, this local holonomy may not be extended to a global holonomy. 


For the partially hyperbolic diffeomorphisms Hu, Hua and Wu (\cite{HHW}) proved a simpler formula for the unstable entropy which coincides with the  definition using increasing subordinated partitions. Fix a small $\epsilon > 0$ and let $\mathcal{P}$ be the set of all finite measurable partitions of $M$ with diameter less than $\epsilon$. Recall that diameter of a partition is the supremum of the diameter of its elements. Given any partition $\alpha \in \mathcal{P}$ one may define a finer partition $\beta$ such that $\beta(x) := \mathcal{F}_{loc}^u(x) \cap \alpha(x)$. Using  notations from \cite{HHW}, we define $\mathcal{P}^u$ as the set of partitions obtained in this way from partitions in $\mathcal{P}$. Let $\mathcal{Q}^u$ be the set of  increasing partitions subordinated to the unstable foliation $\mathcal{F}^u$ defined as in \ref{entropyalong}.

Given a partition $\alpha$ and $m \leq n \in \mathbb{Z}$, let us denote by $\alpha_{m}^{n} := \bigvee_{i= m}^{n} f^{-i} \alpha$ and  define the dynamical conditional entropy of $\alpha$ given $\beta$:
$$
h_{\mu}(f, \alpha | \beta) = \limsup_{n \rightarrow \infty} \frac{1}{n} H_{\mu} (\alpha_{0}^{n-1} | \beta).
$$
The conditional entropy of $f$ given $\eta \in \mathcal{P}^u$ is defined as 
$$h_{\mu}(f| \eta) = \sup_{\alpha \in \mathcal{P}} h_{\mu}(f, \alpha | \eta)$$
and finally unstable metric entropy of $f$ is defined as 
$$ h^u_{\mu}(f):= \sup_{\eta \in \mathcal{P}^u} h_{\mu}(f|\eta).$$ In fact the above definition is independent of $\eta$ as long as $\eta$ is chosen in $\mathcal{P}^u.$
 
A  main result of \cite{HHW} is the following equivalence of definitions for unstable entropy.
\begin{theorem} \cite{HHW}
Suppose $\mu$ is an ergodic measure. Then for any $\mathcal{\alpha} \in \mathcal{P}, \eta \in \mathcal{P}^u$ and $\xi \in \mathcal{Q}^u$,
$$
h_{\mu}(f, \alpha | \eta) = h_{\mu}(f, \xi)
$$ 
\end{theorem}

\begin{remark}
It is not clear whether a similar result holds also for the unstable entropies in the non-uniformly hyperbolic case.
\end{remark}

By Yang result \cite{Yang}, if $\mathcal{F}$ is an expanding foliation invariant by a $C^1$ diffeomorphism $f : M \rightarrow M$ then for any invariant measure $\mu,$ there is a measurable partition subordinated   to $\mathcal{F}$ and $h_{\mu}(f, \mathcal{F})$, the entropy along $\mathcal{F}$ is well defined. By Oseledets' theorem, one can consider Lyapunov exponents along the tangent bundle of $\mathcal{F}$ and the integrated sum of the Lyapunov exponents along $\mathcal{F}$ is denoted by $\lambda^{\mathcal{F}}(f, \mu)$ which is also equal to $\int \log (\det (Df|T\mathcal{F}^{f})) d\mu.$ 

We recall that Saghin-Yang \cite{SaY} proved Margulis-Ruelle inequality and Pesin entropy formula for general expanding foliations. 

\begin{definition} \label{e-gibbs}
Suppose $\mathcal{F}$ is an expanding foliation for a $C^1$ diffeomorphism $f.$ An $f-$invariant probability $\mu$  is called {\it Gibbs expanding state along $\mathcal{F}$} if the entropy of $\mu$ along $\mathcal{F}$ coincides with the sum of exponents along the tangent to $\mathcal{F}$:
$$
 h_{\mu}(f, \mathcal{F}) = \int \log |\det(Df | T_x \mathcal{F})| d\mu(x).
$$

 The collection of all Gibbs e-state along $\mathcal{F}$ is denoted by $Gibbs^e(f , \mathcal{F}).$

\end{definition}
\begin{remark}
If $f$ is $C^r, r > 1$ the above definition is the same as absolute continuity (with respect to Lebesgue) of disintegration of $\mu$ along the leaves of $\mathcal{F}.$ This is justified in the following theorem. 
\end{remark}

\begin{theorem} \label{SaYtheorem}
Let $\mathcal{F}$ be an expanding foliation with $C^r$ leaves of a $C^r$ diffeomorphism $f, r \geq 1$ and $\mu$ an invariant measure for $f.$ Let $\lambda^{\mathcal{F}}(f, \mu)$ be the integrated sum of Lyapunov exponents along the tangent bundle of $\mathcal{F}.$ Then:
$$
 h_{\mu} (f, \mathcal{F}) \leq \lambda^{\mathcal{F}}(f, \mu).
$$
Furthermore, if $r > 1,$ then $\mu$ is a Gibbs e-state along $\mathcal{F}$ if and only if its conditional measures along $\mathcal{F}$ are absolutely continuous with respect to  Lebesgue measure on leaves.
\end{theorem}

\begin{remark}
In this paper, we mostly work with expanding foliations which are strong unstable foliation for a partially hyperbolic diffeomorphism. However, there are some places (like in Section 6) where we consider  weak unstable foliation. For example, if $f: \mathbb{T}^3 \rightarrow \mathbb{T}^3$ is an Anosov automorphism with three distinct eigenvalues $\lambda_1 > \lambda_2 > 1 > \lambda_3$, then the weak unstable foliation is the expanding foliation tangent to eigenspace of $\lambda_2.$ 
\end{remark}

\subsection{Variational principle and $u-$maximal measures} \label{3.2}

Similar to the definition of topological entropy for continuous maps on compact metric spaces, one may define the topological entropy along unstable foliation by means of the notion of spanning (and separated) sets in a family of  compact unstable plaques. A subset $E \subset \mathcal{F}^u(x)$ is called $(n, \epsilon)-$spanning set of $\overline{\mathcal{F}^u(x, \delta)}$ if $\overline{\mathcal{F}^u(x, \delta)} \subset \bigcup_{y \in E} B^u_n(y, \epsilon).$
 Here $B^u_n(y, \epsilon)$ denotes the $u-$Bowen ball around $y$: $ \{z \in \mathcal{F}^u(y) : d^u(f^i(z), f^i(y) \leq \epsilon, 0  \leq i \leq n\}.$ Let $S^u(f, \epsilon, n, x, \delta)$ be the minimal cardinality of $(n, \epsilon)-$spanning set for $\overline{\mathcal{F}^u(x, \delta)}.$
 
 The unstable topological entropy of $f$ is defined by
$$
h^u_{top}(f) = \lim_{\delta \rightarrow 0} \sup_{x \in M} h^u_{top}(f, \overline{\mathcal{F}^u(x, \delta)}) 
$$
where 
$$ h^u_{top}(f, \overline{\mathcal{F}^u(x, \delta)}) = \lim_{\epsilon \rightarrow 0} \limsup_{n \rightarrow \infty} \frac{1}{n} \log S^u(f, \epsilon, n, x, \delta).  $$

It can be proved that  the value of the unstable topological entropy is independent of $\delta.$ That is, it is not necessary to take limit over $\delta.$

\begin{definition} \cite{HSX}
For a partially hyperbolic diffeomorphism $f$ we define volume growth along unstable foliation:
$$ \chi_u(f) = \sup_{x \in M} \chi_u(x, \delta)$$
where $$\chi_u(x, \delta) = \limsup_{n \rightarrow \infty} \frac{1}{n} \log(Vol(f^n(\mathcal{F}^u(x, \delta)))) $$
\end{definition}

It is interesting to note that $h^u_{top}(f) = \chi_u(f).$ This is proved in \cite{HHW} and uses the classical argument of estimating number of separating and spanning sets in the unstable manifolds.

It is also well worth to mention the relation between the volume growth along unstable foliation and the Lyapunov exponents (See Lemma 3.1 in \cite{SX}.): Using the absolute continuity of the unstable foliations (here we need $C^{1+\alpha}$ regularity of $f$), for any absolutely continuous with respect to Lebesgue measure $\mu$, we have that:
\begin{equation} \label{volexp}
\int_M  \log (det Df | E^u(x)) d\mu(x) \leq \chi_u(f).
\end{equation}

By Margulis-Ruelle inequality (along unstable direction, see \cite{WWZ} for unstable entropy version) we have that for any $C^1-$partially hyperbolic diffeomorphism:
\begin{equation} 
\label{entexp}
h^u_{\mu}(f) \leq \int_M \log |det Df|E^u(x)| d \mu(x)
\end{equation}
with equality (in the case of $f \in C^2$) if and only if $\mu$ has absolutely continuous (with respect to Lebesgue) conditional measures along unstable foliation. 
All these together imply that for any $C^{1+\alpha}$ partially hyperbolic diffeomorphism and $\mu$ absolutely continuous with respect to Lebesgue (or just $u-$Gibbs):
\begin{equation} \label{kol}
 h^u_{\mu}(f) \leq \int_M \log |det Df|E^u(x)| d \mu(x) \leq \chi_u(f) = h^u_{top}(f).
\end{equation}

Observe that by Birkhoff ergodic theorem applied to the additive cocycle $x \rightarrow \log \|\wedge^{d_u} D_x f^j| \wedge^{d_u}E^u(x)\|,$ for $d^u = dim(E^u),$ we obtain:
\begin{equation}
\begin{split}
\int \log |det Df|E^u(x)| d \mu(x) =& \int_M \lim_{n\rightarrow \infty} \frac{1}{n} \log |det(Df^n|E^u(x))| d\mu(x)  \\
=& \int_M \lim_{n\rightarrow \infty} \frac{1}{n} \log \| \wedge^{d_u} D_xf^n |\wedge^{d_u}E^u(x) \| d\mu(x)
\end{split}
\end{equation}
where $d_u$ is the dimension of unstable bundle $E^u.$

\begin{question}
Is $C^1-$regularity of $f$ enough to obtain (\ref{volexp}) above?
\end{question}
Finally we recall that Kozlovski \cite{koz} had proved the following relation for the entropy of $C^{\infty}$ diffeomorphisms of a $d-$dimensional compact manifold $M:$
\begin{equation}
 h_{top}(f) = \lim_{n \rightarrow \infty} \frac{1}{n} \log \int_M \|\wedge^d Df_x^n\| dLeb
\end{equation}
where $\wedge^d Df$ stands for the action of $f$ on the full exterior power.

For uniformly expanding maps on compact $d-$dimensional manifolds we have 
\begin{equation}
h_{top}(f) = \limsup_{n \rightarrow \infty} \frac{1}{n} \log \int \|\wedge^d Df^n\| d Leb
\end{equation}

 The celebrated variational principle in ergodic theory, states that for any continuous transformation in a compact metric space, the topological entropy coincides with the supremum over all metric entropies. By the way, in general it may not exist any invariant measure whose  entropy attains the supremum in the variational principle (see \cite{misiu}, \cite{Buzzi}). If $\mu \rightarrow h_{\mu}(f)$ is (upper) semi-continuous function, then  $\mu \rightarrow h_{\mu}$ attains its maximum at some measure(s).  Observe that metric entropy is usually not lower semi-continuous. Indeed, by ergodic closing lemma of Ma\~{n}\'{e} there is a $C^1-$residual subset  of $\Diff(M)$ such that every ergodic measure is approximated by measures supported on periodic orbits.
 
 Both compactness of the phase space and continuity of the dynamics are crucial in the variational principle. However, a variational principle for the entropy along unstable foliation of partially hyperbolic diffeomorphisms can be proved as follows.
\begin{theorem} [\cite{HHW}, Theorem D]
Let $f : M \rightarrow M$ be a $C^1-$partially hyperbolic diffeomorphism. Then
$$
h_{top}^u(f) = \sup \{h^u_{\mu}(f) : \mu \in \mathcal{M}_e(M)\}$$
where $\mathcal{M}_e(f)$ is the set of ergodic invariant measures of $f.$
\end{theorem}

\begin{theorem} [\cite{Yang}, \cite{HHW}]
Let $f \in \Diff^1(M)$ be a partially hyperbolic diffeomorphism and $\{\mu_n\}$ be a sequence of $f-$invariant measures. Assume that $\mu_n$ converge to $\mu$ in weak-* topology, then 
$$
 \limsup_{n \rightarrow \infty} h_{\mu_n} (f, \mathcal{F}^u) \leq h_{\mu}(f, \mathcal{F}^u).
$$
\end{theorem}
By the above semi-continuity result and compactness of $\mathcal{M}(f)$ we conclude that there exists at least one probability measure which maximizes the unstable entropy.
\begin{corollary}
For any partially hyperbolic diffeomorphism there exists an invariant measure which maximizes $u-$entropy.
\end{corollary}

It is well worth to mention that a notion of unstable pressure and unstable equilibrium measures (which is a generalization of $u-$maximal measure) had been defined in \cite{HWZ2}. 
Let us define: $$u-M.M.E (f) = \{\mu ; h^u_{\mu} (f) = h^u_{top}(f)\} $$ and recall that $M.M.E (f)$ is the set of measures of maximal entropy. Here we propose to study the relation between $u-MME(f)$ and $M.M.E(f)$ both subsets of invariant probability measures $f$.

 If $f$ is a partially hyperbolic dynamics and $\mu$ is a $u-$maximizing entropy measure then, $\mu$ may not be a measure of maximal entropy. The trivial example is to consider $f := A \times A^2$ where $A$ is a linear Anosov automorphism of $\mathbb{T}^2$. Clearly, $f$ can be considered as a partially hyperbolic diffeomorphism where the unstable (stable) bundle is the unstable (stable) bundle of $A^2$ and center bundle is two dimensional. Take a fixed point $A(p)=p$ then $\delta_{p} \times Leb_2$ is a measure of maximal $u-$entropy. However, $\mu$ is not a measure of maximal entropy for $f.$

It is interesting to find ``non-trivial" examples of partially hyperbolic diffeomorphisms for which (some) measures of $u-$maximal entropy are not  maximal entropy and $M.M.E (f) \neq \emptyset.$ To find such examples, one may need to understand the distribution of unstable leaves in the manifold. 

\begin{question} \label{rigidity}
Let $f$ be an Anosov diffeomorphism with decomposition $TM = E^s \oplus E^u \oplus E^{uu}$. Is it true that any measure maximizing entropy along strong unstable foliation is a maximal entropy measure for $f$?
\end{question}

We mention that after finishing this survey, R. Ures- M. Viana- F. Yang and J. Yang \cite{UVYY} constructed  examples of partially hyperbolic sets (modified topological solenoid)  where measure of $u-$maximal entropy fails to be measure of maximal entropy. The partially hyperbolic set in their example has $E^u \oplus E^{cs}$ decomposition. However, the diffeomorphism case and the above question is still open.

On the other hand, it is also very interesting to find examples of maximal entropy measures which are not $u-$maximal. In \ref{mmenotu} we give  example of invariant maximal entropy measures which are not $u-$maximal entropy measures.

For a partially hyperbolic diffeomorphism with one dimensional center bundle, an interesting quantity to be understood is the following:
$$
 \sup \{h^u_{\mu} (f), \mu \in m.m.e, \, \, \lambda^c(\mu) > 0\}.
$$

\subsubsection{Entropy of Non-compact sets and Carathéodory dimension}

Similar to Bowen approach, G. Ponce \cite{ponce} and independently  Tian-Wu \cite{TW}  have developed the notion of unstable entropy for non-compact subsets.  As these definitions are in the framework of Carath\'{e}odory dimension, we refer to works of Pesin-Pitskel \cite{PP1} and the reference book of Pesin \cite{Pesin1} for more general definition of Carathéodory construction.  Let $f: M \rightarrow M$ be a $C^1-$partially hyperbolic diffeomorphism. For each $x \in M$ let 
$$M_x:= \bigcup_{j=-\infty}^{\infty} \mathcal{F}^u(f^j(x)).$$

One may define a notion similar to Hausdorff dimension using the dynamics. First, for any open cover $\mathcal{A}$ of $M$ we define a ``diameter"  for any subset $E$ of $ M_x$ by $D_{\mathcal{A}} (E) := e^{-n_{f, \mathcal{A}}(E)}$ where 
$n_{f, \mathcal{A}}(E)$ is the smallest $n$ such that $f^n(E)$ is not a subset of any element of  $\mathcal{A}$.

So, similar to Hausdorff measures we define 
$$
m^x_{\mathcal{A}, \lambda}(Y) := \lim_{\epsilon \rightarrow 0} \inf \{ \sum D_{\mathcal{A}}(E_i)^{\lambda} :    D_{\mathcal{A}}(E_i) \leq \epsilon,  \}
$$
where the infimum is over all possible  families $\{E_i\}_{i=1}^{\infty}$ such that  $Y \subset \bigcup E_i.$

It can be proved that there exists a unique critical value:
$$
h^u_{H, \mathcal{A}}(f, Y) := \inf \{\lambda : m^x_{\mathcal{A}, \lambda}(Y) =0\} ,\,\, Y \subset M_x.
$$  
For $ x \in M , Y \subset M$ define 
$$
h^u_{H}(f, Y, x) := \sup_{\mathcal{A}} h^u_{H, \mathcal{A}}(f, Y \cap M_x)
$$
and finally define $h^u_{H}(f, Y) := \sup_{x \in Y} h^u_H(f, Y, x).$

A main result in \cite{ponce} is that for any ergodic measure $\mu$ and $Y \subset M$ with positive measure then:
$$
h^u_{\mu}(f) \leq h^u_{H}(f, Y)
$$
and $h^u_{H}(f, M) = h^u_{top}(f).$

In \cite{TW} the authors make a detailed analysis of dimension theory too. We recall here the entropy distribution principle for unstable entropy (they call it Bowen unstable entropy and denote by $h_{B} (f, Y)$ for any subset $Y \subset M$). 

\begin{definition}
Let $\mu$ be a Borel probability measure on $M$ and $\xi \in \mathcal{P}^u$ subordinated to unstable foliation of $f: M \rightarrow M$ which is partially hyperbolic. Define:
$$
 \underline{h}^u_{\mu} (f) = \int_M  \underline{h}^u_{\mu} (f, x) d\mu, \,\,  \overline{h}^u_{\mu} (f) = \int_M  \overline{h}^u_{\mu} (f, x) d\mu
$$  
where 
$$ \underline{h}^u_{\mu} (f, x) = \lim_{\epsilon \rightarrow 0} \liminf_{n \rightarrow \infty} -\frac{1}{n} \log \mu_{\eta(x)} (B_n^u(x, \epsilon))$$
$$
\overline{h}^u_{\mu} (f, x) = \lim_{\epsilon \rightarrow 0} \limsup_{n \rightarrow \infty} -\frac{1}{n} \log \mu_{\eta(x)} (B_n^u(x, \epsilon)).
$$
\end{definition}
If $\mu$ is $f-$invariant then 
$$
 h^u_{\mu}(f) = \underline{h}^u_{\mu} (f) = \overline{h}^u_{\mu} (f).
$$
\begin{theorem}
Let $f: M \rightarrow M$ be a $C^1$ partially hyperbolic diffeomorphism. $Y \subset M$ a Borel subset and $\mu$ an $f-$invariant measure, $0 < s < \infty$,
\begin{itemize}
\item If $\underline{h}^u_{\mu} (f, x) \leq s $ for every $x \in Y$ then $h_B^u(f, Y) \leq s$
\item If $\underline{h}^u_{\mu} (f, x) \geq s $ for every $x \in Y$ and $\mu(Y) > 0$ then $h_B^u(f, Y) \geq s.$
\end{itemize} 
\end{theorem}

\subsection{Empirical measures and $u-$Gibbs} \label{3.3}

The notion of $u-$entropy plays a fundamental role in the study of $u-$Gibbs and physical measures. Let $f: M \rightarrow M$ be a homeomorphism of a compact manifold $M.$ Given a point $x\in M$ we denote by $\mathcal{M}(x)$ the set of accumulation points of the empirical measures $\frac{1}{n} \sum_{i=0}^{n-1} \delta_{f^i(x)}$ as $n \rightarrow \infty.$

\begin{theorem} {\normalfont (\cite{HFY} and \cite{CYZ})} \label{generic1}
For any $C^1$ diffeomorphism $f,$ if $\Lambda$ is an attracting set with a partially hyperbolic splitting $E^{cs} \oplus E^u,$ then there exists a small neighborhood $U$ of $\Lambda$ such that for  Lebesgue almost every point $x \in U,$ any limit measure $\mu \in \mathcal{M}(x)$  satisfies
$$
h_{\mu}(f, \mathcal{F}^u) = \int_{\Lambda} \log |det(Df|E^u)| d\mu,
$$
where $\mathcal{F}^u$ is the strong lamination tangent to $E^u.$
\end{theorem}
The equality in the above theorem is the celebrated property of $u-$Gibbs measures for $C^{1+\alpha}, \alpha >0$ partially hyperbolic diffeomorphisms, as defined by Pesin and Sinai (See chapter 11 of \cite{BDV}). We can define $u-$Gibbs measures in the $C^1-$setting, as those which satisfy the equation in the above theorem. A corollary of  theorem \ref{generic1} (mentioned in the same references) is that a unique $u-$Gibbs measure is a physical measure. A measure $\mu$ is called physical if the empirical measures of points in a  Lebesgue positive subset converge to $\mu.$
\begin{corollary}
Let $f$ be a $C^1$ diffeomorphism and $\Lambda$ be an attracting set with a partially hyperbolic splitting $E^{cs} \oplus E^{u}$. Assume that there exists a unique $u-$Gibbs measure $\mu$ on $\Lambda$, then $\mu$ is a physical measure; moreover, its basin has full Lebesgue measure in the topological basin of $\Lambda.$
\end{corollary}

It is well-worth to mention that in \cite{CYZ}, they also construct an example of partially hyperbolic diffeomorphism transitive $f \in \Diff^{\infty}(\mathbb{T}^3)$ with $T\mathbb{T}^3 = E^s \oplus E^c \oplus E^u, \dim(E^c)=1$ such that the Lebesgue almost every point $x \in \mathbb{T}^3$ has dense orbit and its sequence of empirical measures $\frac{1}{n} \sum_{j=0}^{n-1} \delta_{f^j(x)}$ does not converge in weak-* topology.

The following result is a fundamental step to obtain the above results.

\begin{theorem} \label{generic2}
For any $C^1$ diffeomorphism  $f$, for any compact invariant set $\Lambda$ admitting a dominated splitting $E \oplus F$ and for Lebesgue almost every point $x \in M,$ if $\omega(x) \subset \Lambda,$ then each limit measure $\mu \in \mathcal{M}(x)$ satisfies 
$$
h_{\mu}(f) \geq \int \log |det Df|F| d\mu.
$$
\end{theorem}
This theorem in this version has been proved in the appendix (A) of \cite{CYZ}.  It appeared firstly in the work of \cite{CCE}, Theorem 1. See also \cite{CaYa} Theorem 4.1.

It is also well-worth to mention a result of Sun and Tian where a Pesin entropy formula is obtained in $C^1-$setting (see also \cite{Ta} for the two dimensional case):
\begin{theorem}
Let $f: M \rightarrow M$ be a  $C^1$ diffeomorphism on a compact Riemannian manifold. Let $f$ preserve an invariant probability measure $\mu$ which is absolutely continuous with respect to Lebesgue measure. For $\mu$ a.e. $x \in M$ denote by $$\lambda_1(x) \geq \lambda_2(x) \geq \cdots \geq \lambda_{dim(M)}(x)$$
the Lyapunov exponents at $x$. If for $\mu$ a.e. $x \in M$ there is a dominated splitting $T_{orb(x)}M = E_{orb(x)} \oplus F_{orb(x)},$ then 
$$ h_{\mu} (f) \geq  \int \log |det Df|F(x)| d \mu(x).$$
\end{theorem}
\subsection{Robust transitivity mechanism} \label{3.4}
We would like to  mention a new mechanism (due to J. Yang \cite{Yang}) to obtain robust transitivity of diffeomorphism. As it uses $u-$entropy and is interesting by itself we recall some main points of this construction.

Let
 \begin{equation} G^u(f) := \{\mu \in \mathcal{M}_{inv}(f) : h_{\mu}(f, \mathcal{F}_f^u) \geq \int \log(\det(Df|E^u(x))) d\mu(x)\}
 \end{equation}
 and 
 \begin{equation}
 G^{cu}(f) := \{\mu \in \mathcal{M}_{inv}(f) : h_{\mu}(f) \geq \int \log(\det(Df|E^{cu}(x))) d\mu(x)\}
 \end{equation}

By Ledrappier \cite{Ledrappier84}, if $f$ is $C^{1+}$ we have $G^u(f)$ coincides with the set of $u-$Gibbs measures, i.e; invariant measures with absolutely continuous with respect to Lebesgue conditional measure along unstable foliation.

\begin{theorem}
let $f$ be a $C^{1+}$ volume preserving partially hyperbolic diffeomorphism, accessible and with one dimensional center bundle and non-vanishing center Lyapunov exponent. Then, $f$ is $C^1$ robustly transitive. That is, every $C^1$ close diffeomorphism to $f$ is topologically transitive.
\end{theorem}

\begin{remark}
Observe that in particular using perturbation given by Baraviera-Bonatti \cite{babo} for removing vanishing exponent, one can perturb time-one map  of a mixing volume preserving Anosov flow to obtain $f$ satisfying the above theorem. In this way, we obtain robustly transitive diffeomorphisms accumulating on the time-one map of the flow. 

\end{remark}
We give a sketch of the proof of the above theorem.
\begin{proof}
We known that $f$ is ergodic (\cite{HHU1}, \cite{BW1}). Recall that ergodicity of volume measure (or any fully supported measure) implies topological transitivity. However we do not need to prove ergodicity, as any volume preserving partially hyperbolic diffeomorphism with accessibility property is topologically transitive. Even more, almost every obrit is dense. 
Without loss of generality we assume that center exponent of $f$ with respect to volume is positive. 

We define a new subset of invariant measures $$G(g): = G^u(g) \bigcap  G^{cu}(g)$$ for $g$ in a small neighbourhood $\mathcal{U}$ of $f$ which plays a crucial toolbox role in the proof of topological transitivity. In fact the result gives slightly stronger than transitivity property. The main property used here is that, 
$$g \in \mathcal{U} \rightarrow G(g),$$
is an upper semi-continuous function. It needs a proof that we refer to the original paper \cite{Yang}. It is also proved that  for any $C^{1+}$ diffeomorphism in a $C^1$ neighborhood of $f$, $G(g)$ is the  unique physical measure. As $f$ is volume preserving, $G(f) = \{vol\}$. 

By semi-continuity we know that $G(g)$ is close to $\{vol\}.$ As center exponent of $f$ is positive, this will imply (not immediately) that for  some constant $b >0$, any $g \in C^1$ close enough to $f$ in $C^1$ topology and $\mu \in G(g)$
$$
\lambda^c_{\mu}(g) > b. 
$$
 Using theorems \ref{generic1} and \ref{generic2} (intersecting two full Lebesgue measure sets of the two theorems) we obtain that for a full Lebesgue measure subset $\Gamma_{g}$ and $x \in \Gamma_g:$
$$
 \liminf_{n \rightarrow \infty} \frac{1}{n} \sum_{i=0}^{n-1} \log \|Df|E^c(g^i(x))\| > b >0.
$$
This lower bound is used to show that for a positive lower density sequence of numbers $n \in \mathbb{N}$, for any $x \in \Gamma_g$ ``the disks based on $x$ expand". 

Let $U, V$ be two arbitrary open subsets of the ambient manifold. So, taking $x \in \Gamma_g \cap U$ and a disk $D \subset U$ tangent to center-unstable cone-field and centered at $x$ for infinitely many $n$ (positive density) $D_n:=g^n(D)$ contains a uniformly large disk. Saturating this large disk with local stable leaves we obtain a cylinder called $B_g$ with uniformly (in $g$) positive volume measure, i.e $vol(B_g) \geq \alpha_0$. 

Now the next step is to show that 
for almost every point  $y \in V,$ and for infinitely many $m \in \mathbb{N}, g^{-m}(y) \in B_g$. As the negative iterates expand the stable foliation we conclude that for a small stable plaque  $K^s \in V$ containing $y$ for large $m, g^{-m}(K^s) \cap D_n \neq \emptyset$. We have proved that some negative iterate of $V$ intersects some positive iterate of $U$ which yields topological transitivity of $g.$

A fundamental tricky lemma (to prove the above step) by Yang  is that  there exists $p_0 > 0$ such that any $\mu \in G(f^{-1})$ can be written as $p vol + (1-p)\mu^+$ where $p \gneqq p_0$ and all ergodic components of $\mu^+$ are $u-$Gibbs for $f^{-1}$ with positive center Lyapunov exponent. 
Cover the ambient manifold with finitely many balls $\{B_1, \cdots, B_k\}$ with fixed small radius such that for every $g,$ close enough to $f$ there exists $B_i$ such that $B_i \subset B_g.$ In particular for any $\mu \in G(f^{-1})$ we have $\mu(B_i) > p_0 \alpha_0$ where $\alpha_0= Min \{ vol(B_i), i=1, \cdots,k\}.$ By continuity argument the same holds for any $\mu \in G(g^{-1})$ for $C^1-$close diffeomorphisms $g.$ There is also a neighbourhood $\mathcal{V}$ of probabilities (not necessarily invariant for f) of $G(f^{-1})$ such that for any $\nu \in \mathcal{V}$ and any $1 \leq i \leq k$, 
$$ 
\nu(B_i) > p_0 \alpha_0.$$

Using Theorems \ref{generic1} and \ref{generic2} (for $g^{-1}$) we have a full measurable subset $\Gamma(g^{-1})$ such that for any $x \in \Gamma_{g^{-1}}$ accumulation points of empirical measures of $x$ are inside $G(g^{-1})$ and so there exists $n = n_x$ such that:
$$
 \frac{1}{n} \sum_{i=0}^{n-1} \delta_{(g^{-i}(x))} \in \mathcal{V}.
$$
Consequently, for every $B_i$ we have 
$$
\frac{1}{n} \sum_{i=0}^{n-1} \delta_{(g^{-i}(x))} (B_i) > p_0 \alpha_0
$$
So, taking  $x \in \Gamma(g^{-1}) \cap V$ the orbit of $x$ by $g^{-1}$ enters all $B_i$  and in particular enters $B_g$. 

\end{proof}

\section{Invariance Principle} \label{invprincple}
In this section we review a celebrated theorem of Furstenberg about the Lyapunov exponents of random product of matrices. This is a particular case of linear cocycles which itself can be seen as an special case of random walk on the group of diffeomorphisms of the projective space.

In the first subsection we give some details about the Furstenberg result and its generalization by Ledrappier \cite{L} and address the non-linear version of all these results by Avila-Viana \cite{AV}.  

In the second subsection, we prove a  result  joint with J. Yang \cite{TY} where we give a new criterion for the invariance principle using the notion of unstable entropy.  This criterion shed light on the proof of  previous results in invariance principle and has several applications.

\subsection{Linear cocycles}The simplest and best model to discuss the invariance principle is the setting of random product of matrices. This is a special case of linear cocycles. Let us state the Furstenberg theorem. Let $(M, \mathcal{B}, m)$ be a probability space and $f: M \rightarrow M$ a measure preserving map. Let $A: M \rightarrow SL(d)$ be a measurable function with values in the linear special group. The linear cocycle defined by $A$ is 
$$F: M \times \mathbb{R}^d \rightarrow M \times \mathbb{R}^d, \, \, F(x,v) = (f(x), A(x)v).$$

For any $n > 0$ we have
$$
F^n(x, v) = (f^n(x), A^n(x)(v)),
$$
where $A^n(x) := A(f^{n-1}(x)) \cdots A(f(x))A(x).$

If $f$ is invertible then $F$ is also invertible and $F^{-n}(x, v) = (f^{-n}(x), A^{-n}(x)v)$ where $A^{-n}(x):= A(f^{-n}(x))^{-1} \cdots A(f^{-1}(x))^{-1}.$

By Furstenberg-Kesten theorem, if $\log \|A^{\pm}(x)\| \in L^1(\mu)$ then the top Lyapunov exponents exist almost everywhere:
$$
\lambda^{+} = \lim_{n \rightarrow \infty} \frac{1}{n} \log(\|A^n(x)\|) \quad \text{and} \quad \lambda^{-} = \lim_{n \rightarrow \infty} \frac{1}{n} \log(\|A^n(x)^{-1}\|^{-1})
$$
By definition $\lambda^+ \geq \lambda^{-}$.
In the special case of random product of matrices, $(M, m) = (SL(d)^{\mathbb{Z}}, \nu^{\mathbb{Z}})$ where $\nu$ is a probability measure on the group of special matrices $SL(d),$  we define $A(x) =x_0$ where $x = (x_i)_{i \in \mathbb{Z}}.$ In the non-invertible case, one puts $M =  SL(d)^{\mathbb{N}}$. It is clear from the definition that 
$$
 F^n(x, v) = (f^n(x), x_{n-1}\cdots x_1 x_0 v).
$$ 
Observe that the second coordinate of the right-hand side of the above formula is the action of the product of $n$ random matrices chosen with the distribution $\nu$ on the vector $v.$

A fundamental project in the theory of random matrices is to understand whether $\lambda^+ >0$ or not.  Let us concentrate on the case of $SL(2)$ where $\lambda^{+} + \lambda^{-} =0$ and consequently either $\lambda^+>0>\lambda^{-}$ or $\lambda^+ = \lambda^- = 0.$   
The following theorem is due to Furstenberg:
\begin{theorem} \label{furst}
Consider the random product of matrices in $SL(2)$ as above and assume that
\begin{enumerate}
\item the support of $\nu$ is not contained in a compact subgroup of $SL(2)$ and 
\item there is no non-empty finite subset $L \subset \mathbb{P}\mathbb{R}^2$ invariant by every $A$ ($A(L) =L$) in the support of $\nu,$
\end{enumerate}
then $\lambda^+ > 0 > \lambda^{-}.$

\end{theorem}

To prove the above theorem we observe that the hypotheses of the theorem imply that there is no probability measure $\eta$ on $\mathbb{P}\mathbb{R}^2$ and invariant by all $A \in \supp(\nu).$ Then applying an ``invariance principle" (proposition \ref{invmeasure}) the proof is complete.
Indeed, if there exists an invariant measure $\eta$ either it has atoms or it is a diffuse measure (without atoms). If the measure is diffuse then an elementary linear algebra exercise shows for any $A \in \supp(\nu)$  $\|A\| = 1$  and consequently the support of $\nu$ is in a compact subset of $SL(2).$ In fact it is possible to show that the support is inside the rotation subgroup after a conjugacy.
If $\eta$ has an atom then the set of atoms (which is finite) is an invariant set by all $A \in \supp(\nu).$

 We may substitute the second condition with: There is no $L$ such that $L$ is invariant by every element $B$ in the smallest sub-group generated by $\supp(\nu).$

\subsubsection{Algebraic comment}
Observe that in the above theorem one can substitute the hypotheses equivalently with strong irreducibility and proximality condition on the semi-group generated by the $\supp(\nu)$.  We say that a matrix $g$ in $GL(\mathbb{R}^n)$ is proximal, if there exists a vector $v \in \mathbb{R}^n, g(v) = \lambda v$ and $\mathbb{R}^n = \mathbb{R}v \oplus W$ such that $g(W) \subseteq W$ and $r(g|W) < \lambda$ where $r(.)$ represent spectral radius. 
A semi-group is called proximal if it contains a proximal element. In fact under strong irreducibility condition, the proximality of semi-group and group is the same, as long as we are considering $\mathbb{R}$ as the field where the matrices are defined.

So, the proof of the Furstenberg theorem is reduced to an ``invariance principle":

\begin{proposition}  \label{invmeasure}
\label{furstinv} In the setting of Theorem \ref{furst}  if $\lambda^+ =  \lambda^{-}$ then there exists a probability measure $\eta$ invariant by all elements in $\supp(\nu).$
 \end{proposition}

The above proposition is  corollary of the following result of Ledrappier.
\begin{theorem} \cite{L}
Let $(M, \mathcal{B}, m)$ be a probability space and $f$ is $m-$preserving. Let $A: M \rightarrow SL(d)$ be a measurable function. Suppose $\mathcal{B}_0$ is a generating decreasing ($f^{-1}(\mathcal{B}_0) \subseteq \mathcal{B}_0$) sub-$\sigma-$algebra. Let $\mu \in \mathcal{M}(F)$ which projects on $m.$  Suppose that $\lambda^+ = \lambda^{-}$. If $A$ is $\mathcal{B}_0$ measurable then
$$
x  \rightarrow \mu_{x} ; \, \, x \in M
$$ is also $\mathcal{B}_0$ measurable where $\{\mu_x\}$ is the disintegration of $\mu$ along the fibers.
\end{theorem} 

 However, the above theorem is an special case of a more general result (non-linear invariance principle) by Avila-Viana \cite{AV}. Here we follow this non-linear approach.

For this purpose we consider the projectivization of the cocycle $\mathbb{P}F : M \times \mathbb{P}\mathbb{R}^2 \rightarrow  M \times \mathbb{P}\mathbb{R}^2$ which is defined naturally as $\mathbb{P}F  (x, [v]) = (f(x), [A(x)v]).$ Observe that $M \times \mathbb{P}\mathbb{R}^2 $ is fibered over $M$ by $\pi: M \times \mathbb{P}\mathbb{R}^2 \rightarrow M$ and fibers are compact one dimensional manifold $\mathbb{P}\mathbb{R}^2$. Moreover, $\mathbb{P}F$ sends fibers to fibers. This is a non-linear cocycle as in \cite{AV}.     

By $\mathcal{M}_{m}(\mathbb{P}F)$ we denote the set of probability measures on $M \times \mathbb{P}\mathbb{R}^2$ which are $\mathbb{P}F$ invariant and project on $m$ by $\pi.$
 By Oseledets' theorem for $m-$almost every $x \in M$ and every $v \in \mathbb{R}^2 \setminus \{0\} $, $\lim_{n \rightarrow \infty} \frac{1}{n} \log\|A^{(n)}(x)v\|$ exists. 

As an exercise of calculus one obtains the derivative of the projectivized cocycle along fibers:
$$
D \mathbb{P}F^n_{(x, v)}(w) = \frac{Pr_{(A^n(x)(v))^\perp}(A^n(x)(w))}{\|A^n(x)(v)\|}
$$
where $Pr_{b^{\perp}}(a) = a - \frac{a.b}{\|b\|^2}b$ is the projection of $a$ on the orthogonal of $b.$ 

Indeed, if $A: \mathbb{R}^n \rightarrow \mathbb{R}^n$ and $\mathbb{P}A : \mathbb{P}\mathbb{R}^n \rightarrow  \mathbb{P}\mathbb{R}^n$ its projectivization, then $\mathbb{P}A (v) = \frac{Av}{\|Av\|}$ and so $Av = \|Av\| \mathbb{P}A(v)$. Taking derivative from both sides and applying on $w \in T_v \mathbb{P}\mathbb{R}^n$ we have $Aw = D_v n (w) . \mathbb{P}F(v) + \|Av\| D\mathbb{P}F_v(w)$ where $n(v):= \|A(v)\|.$ We have $n^2(v) = <Av, Av> $ and so $D_{v}n(w) = \frac{<Av, Aw>}{\|Av\|}$ and putting these together we obtain
$$ D_v\mathbb{P}A (w) = \frac{Aw - \frac{<Av, Aw>}{\|Av\|^2 }.Av}{\|Av\|} = \frac{Pr_{(Av)^{\perp}}Aw}{\|Av\|}$$

For the cocycle we apply this argument for $A^n.$ As a corollary we have:
$$
\|D \mathbb{P}F^n_{(x, [v])}(w)\| \leq \frac{\|A^n(x)w\|}{\|A^n(x)v\|}.
$$
Clearly here by $D \mathbb{P}F^n_{(x, [v])}$ we mean the derivative of $\mathbb{P}F^n$ along the smooth direction, which is the fiber.
Repeating the above argument changing $n$ to $-n$ it comes out that if $\lambda^+=\lambda^{-} $ then for $m-$almost every $x$ we have 
$$\lim_{n \rightarrow \infty} \frac{1}{n} \log \|D \mathbb{P}F^n_{(x, v)}(w)\| = 0.$$ In other words for {\it every} measure $\mu \in \mathcal{M}_{m}(\mathbb{P}F)$ the Lyapunov exponent along the fiber is vanishing almost everywhere.

Let $\mathcal{B}_0$ be the $\sigma-$algebra generated by all local stable sets in $SL(2)^{\mathbb{Z}},$ i.e $W^s_{loc}(x) = \{y | y_i = x_i, i \geq 0\}$ and their iterates $f^{-n} (W_{loc}(x)), n\geq 0$. As the coycle is locally constant we conclude that $x \rightarrow A(x)$ is 
$\mathcal{B}_0$ measurable. $\mathcal{B}_0$ is a generating $\sigma-$algebra. As a corollary of Avila-Viana non-linear invariance principle (Theorem B in \cite{AV}), for a $\mathbb{P}F-$ invariant measure $\mu \in \mathcal{M}_{m}(\mathbb{P}F)$ with $\lambda^c$ being the Lyapunov exponent along the fibers we have the following result: Let   $x \rightarrow \mu_x$ be a system of disintegration of $\mu$ along fibers, then:
\begin{itemize}
\item  If $\lambda^c \geq 0$ then $x \rightarrow \mu_x$ is $\mathcal{B}_0-$measurable.
\item If $\lambda^c \leq 0$ then $x \rightarrow \mu_x$  is measurable with respect to unstable subsets. (just considering $F^{-1}$ instead of $F$.)
\item If $\lambda^c = 0$ then $x \rightarrow \mu_x$ is constant $m-$almost everywhere on $M.$
\end{itemize}
Observe that the third conclusion comes from the first two ones and the fact that the intersection of stable and unstable $\sigma-$algebras is the trivial $\sigma-$algebra $\{\emptyset, \Sigma\}$ and measurability with respect to trivial $\sigma-$algebra is exactly the conclusion of the this item. As $m$ has a local product structure, then a  Hopf type argument shows that one can extend continuously $x \rightarrow \mu_x$ for all $x \in M$ as a constant function $\mu_{x_0}$. 

Now, as $\mu$ is $\mathbb{P}F-$invariant and conditional measures are defined uniquely we conclude that $A(x)_* \mu_x = \mu_{f(x)}$ for $m-$almost every $x.$ This yields that $\nu-$ almost every matrix preserves $\mu_{x_0}$ and consequently every matrix in the $\supp(\nu)$ preserves the same measure. The proof of proposition (\ref{furstinv}) is complete.

In the proof of the above results (invariance principle of Ledrappier and Avila-Viana), a notion of entropy called Kullback-Leibel information is hidden. In fact the vanishing exponent implies that some ``entropy" is vanishing. In the general setting of linear cocycles Ledrappier proved an upper bound for the Kullback-Leibel information in terms of Lyapunov exponents along fibers of the projectivized cocycle (Proposition 4 in \cite{L}). A similar bound holds in the non-linear case, as proved in a non-trivial generalization of Ledrappier result by Avila-Viana. See also Crauel \cite{crauel}.

\subsection{A general entropy criterion to obtain invariance principle} \label{general}
 Let $(M, \mathcal{B}, m) $ be Borel standard space and $\sigma$ a measurable partition that satisfies  conditions (b) and (d) in the subsection \ref{entropyalong}.
 Usually, $\sigma$ is  subordinated to some unstable lamination, coming from Pesin theory or uniformly exanding foliations.

Now let $F: \mathcal{E}=M \times N \rightarrow M \times N$ be a measurable fiber bundle map with fibers modeled by $N$  a Borel standard space. 
Observe that when $M$ and $N$ are Borel standard spaces then, $\mathcal{E} = \bigcup_{n \in N} M \times \{n\}$ and $\mathcal{E} = \bigcup_{m \in M} \{m\} \times N$, respectively horizontal and vertical  partitions of $\mathcal{E}$  are measurable.

 We assume that $F$  covers $M$ and the following diagram commutes:
$$\begin{array}{ccc}
 \mathcal{E}& \xrightarrow{F} & \mathcal{E}\\
\downarrow{\pi}  & & \downarrow{\pi} \\
M & \xrightarrow{f} & M
\end{array}$$

Define $\xi^{cu}:= \pi^{-1}(\sigma).$ 
As $F$ is fiber preserving we conclude that $\xi^{cu}$ is a measurable partition of $\mathcal{E}$ with respect to any measure $\mu$ which projects on $m,$ i.e, $\pi_* \mu = m.$ 

 By definition all elements of this partition are saturated by fibers and satisfy $F^{-1}(\xi^{cu}(x)) \subset \xi^{cu}(F^{-1}(x)).$

Now suppose that there exists another increasing measurable partition $\xi^u$ for $\mathcal{E}$ ($F \xi^u < \xi^u$) such that 
\begin{enumerate}
\item $\pi^{-1}(\sigma)= \xi^{cu} < \xi^u$
\item   $\pi(\xi^u(x))) = \sigma(\pi(x))).$
\\
In general $\xi^u$ produces a holonomy map $h^u$ between two fiber in the same atom of $\xi^{cu}.$
For simplicity we assume that:

\item  $h^u$ is injective and surjective between two fibers in the same atom of $\xi^{cu}.$

\end{enumerate} 
In what follows, let $\mu$ be an $F-$invariant measure where $\pi_* \mu = m.$ We fix some notations about the conditional measures of $\mu$ along different measurable partitions introduced above.
\begin{itemize}
\item $\mu^u$ denotes the conditional measure of $\mu$ along the measurable partition $\xi^u.$
\item $\mu^{cu}$ represents conditional measure of $\mu$ along the measurable partition $\xi^{cu}$ and $ ( \mathcal{E}/\xi^{cu}, \tilde{\mu})$ denotes the corresponding quotient space of $\mathcal{E}$ by the atoms of the partition $\xi^{cu}$. So, $\mu = \displaystyle \int \mu^{cu} d \tilde{\mu}$ and by uniqueness of disintegration in Rokhlin theorem,  $\mu^{cu} = \displaystyle \int \mu^u d\mu^{cu}.$ Observe that in the last formula we used the disintegration formula without quotient measure.
\item  $(M/\sigma, \tilde{m}) $ is the quotient space of $M$ into partition $\sigma$ with probability quotient measure $\tilde{m}$, i.e $m= \displaystyle \int m^u d \tilde{m}.$ 
\item Observe that $\pi$ induces a natural isomorphism between $ ( \mathcal{E}/\xi^{cu}, \tilde{\mu})$  and $(M/\sigma, \tilde{m})$. In what follows, in order to reduce notations, we do not distinguish between these two spaces, always having in mind their isomorphism. 
  
\end{itemize}
\subsubsection{Holonomy invariance}

Let $\mu$ be an $F$ invariant measure such that $\pi_*\mu = m$ as in  \ref{general}. We define the following holonomy invariance notions inside each atom of $\xi^{cu}:$
Given two points $x, y \in \mathcal{E}$ in the same atom of $\xi^{cu},$ let $ x = (m_x, n_x), y = (m_y, n_y )$. We define the center holonomy $H^c_{x,y}: \xi^u(x) \rightarrow \xi^u(y)$ by $H^c_{x,y} (m, n_x) := (m, n_y).$
\\
Similarly one defines another holonomy $H_{x,y}^u$ between fibers by $H^u_{x,y} (m_x, n) = (m_y, n).$

\begin{definition}
Let $\mu$ be an $F-$invariant measure, then $\mu$ is called invariant by the holonomy defined by fibers, if there exists a $\mu-$full measurable subset such that for any two points $x, y$ in it belonging to the same element of $\xi^{cu}$ we have $ \mu_y^u = (H_{x,y}^c)_* \mu_x^u.$ This is equivalent to say that for $\mu$ almost every $x,$
$\pi_* \mu_{x}^u = m_{\pi(x)}^u.$ 
\end{definition}
Similarly we define the notion of $u-$invariance:
\begin{definition}
  Let $\mu$ be an $F-$invariant measure, then $\mu$ is called $u-$invariant, if there exists a $\mu-$full measurable subset such that for any two points $x, y$ in it and belonging to the same atom of $\xi^{cu}$ we have $(H^u_{x,y})_* \mu^c_x = \mu^c_y.$ 
\end{definition}  
  
   By Lemma \ref{product}  we  have the equivalence of $u-$invariance and fiber invariance.

\begin{lemma} \label{quotient}
$\pi_* (\mu_{x}^{cu}) = m^u_{\pi(x)}$ for $\mu$ almost every $x \in \mathcal{E}.$
\end{lemma}

\begin{proof}

By definition $\mu = \displaystyle{\int} \mu^{cu} d \tilde{\mu}$ where $\tilde{\mu}$ is the probability on the quotient $\mathcal{E}/ \xi^{cu}$ which is identified with $M/\sigma.$ We also have
$m = \displaystyle{\int} m^u d \tilde{m}$ where the quotient measure $\tilde{m}$ is defined on  $M/\sigma .$  As we wrote before we may consider (by an abuse of notation) $\tilde{m} = \tilde{\mu}.$ Indeed taking any measurable subset $S \subset  \mathcal{E}/\xi^{cu}$ we have:
$$
 \tilde{\mu} (S) = \mu( \bigcup_{t \in S} \xi^{cu}(t))
$$
and
$$
\tilde{m} (S) = m(\bigcup_{t \in S} \sigma(t)) = \mu( \pi^{-1} (\bigcup_{t \in S} \sigma(t)) ) = \mu( \bigcup_{t \in S} \xi^{cu}(t)).
$$

So, by definition
\begin{align*}
\int_{M/\sigma } m^u d \tilde{m} = m =& \pi_*{\mu}
= \pi_* (\displaystyle{\int}_{\mathcal{E}/ \xi^{cu}}\mu^{cu} d \tilde{\mu} ) \\=& \int_{\mathcal{E}/ \xi^{cu}} \pi_{*} \mu^{cu}  d \tilde{\mu} = \int_{M/\sigma } \pi_{*} \mu^{cu}  d \tilde{m}
\end{align*}

By essential uniqueness of disintegration we conclude the proof of lemma.
\end{proof}

\begin{theorem}\label{criterion} \cite{TY} (Entropy criterion)
 $h_{\mu}(F, \xi^u) \leq h_{m}(f, \sigma)$ and the equality holds if and only if the disintegration of $\mu$ along the fibers is invariant under the holonomy defined by $\xi^{u}.$
\end{theorem}
\begin{figure}
\includegraphics[scale=0.5]{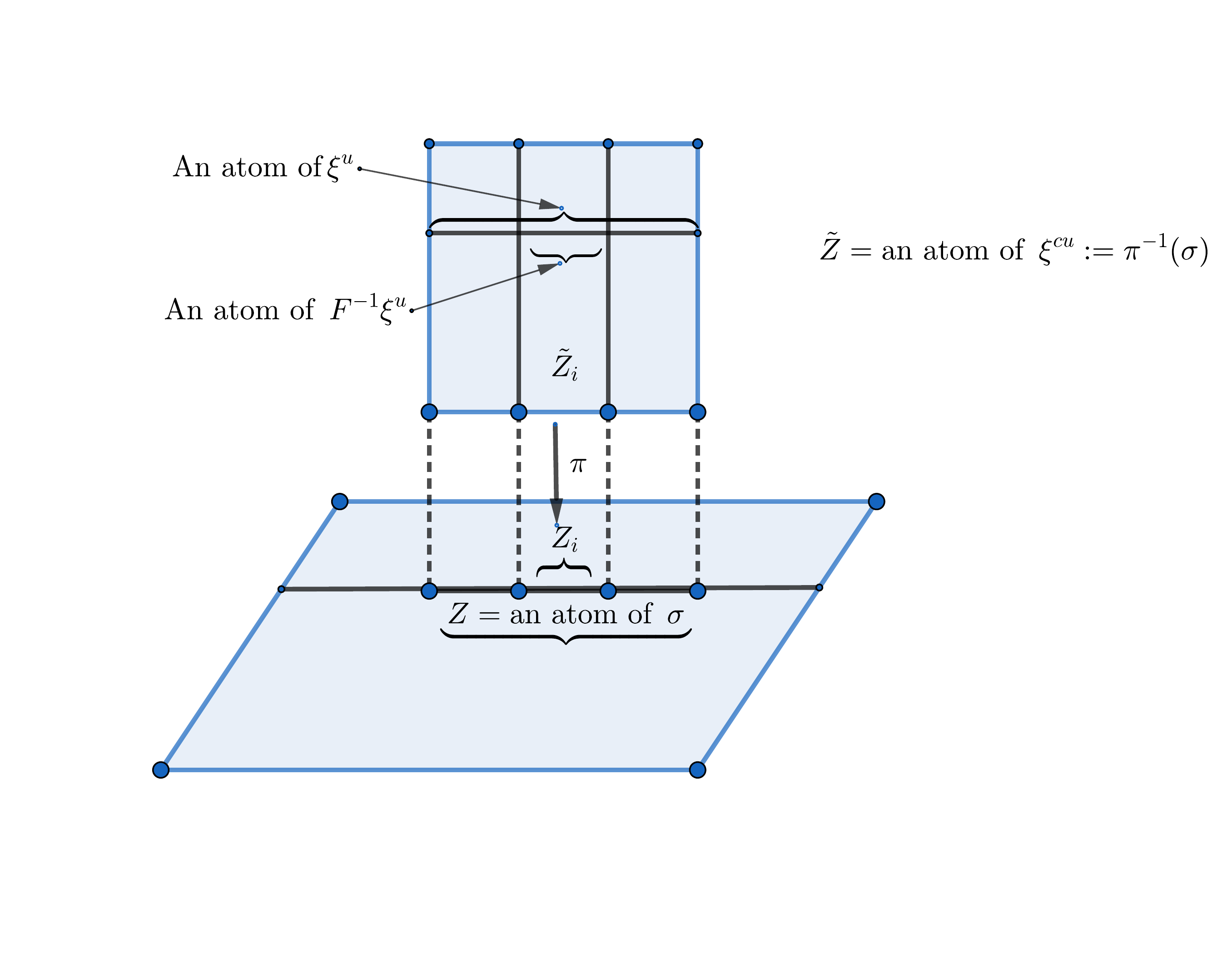}
\centering
\caption{}
\end{figure}

\begin{proof}
By definition 

\begin{equation}
\begin{split}
h_{\mu}(F, \xi^u)  = H_{\mu}(F^{-1}\xi^u | \xi^u) &= \displaystyle \int_\mathcal{E} \log  
\mu^u_z ( (F^{-1}\xi^u)(z)) d \mu(z) \\
&=\displaystyle \int_{\mathcal{E}/\xi^{cu}} \int_{\tilde{Z}} -\log  
\mu^u_t ( F^{-1}\xi^u(t)) d \mu_{\tilde{Z}}^{cu}(t) d\tilde{\mu}(\tilde{Z}) \\
&= \displaystyle \int_{M/\sigma} \int_{\pi^{-1}(Z)} -\log  
\mu^u_t ( F^{-1}\xi^u(t)) d \mu_{Z}^{cu}(t) d\tilde{m}(Z) 
\end{split}
\end{equation}
In the second equation we used the disintegration of $\mu$ into its conditional measures along the elements of the partition $\xi^{cu}$ and
in the last equation we used the identification between $Z \in M /\sigma$ and $\tilde{Z} \in \mathcal{E}/\xi^{cu}.$ 
On the other hand:
$$
h_{m}(f, \sigma) = - \displaystyle \int_{M} \log m_x^u((f^{-1} \sigma) (x)) d m(x) = - \displaystyle \int_{M/\sigma} \int_{Z}\log m_Z^u(f^{-1}\sigma (s))  dm_Z^u(s) d \tilde{m}(Z)
$$

So we have written both $h_{\mu}(F, \xi^u)$ and $h_{m}(f, \xi)$ as double integrals. We compare the inner integrals fixing $Z.$

\begin{lemma} 
For $\tilde{m}-$almost every $Z$ (and $\tilde{\mu}-$almost every $\tilde{Z}, \pi(\tilde{Z})= Z.$), we have 
$$
 \int_{\tilde{Z}} - \log \mu_t^u(F^{-1}\xi^u (t))  d \mu_{\tilde{Z}}^{cu}(t)  \leq  \int_{Z} - \log m_Z^u(f^{-1}\sigma (s))  dm_Z^u(s).
$$ 
\end{lemma}
 Observe that for almost every  element $Z$ of $\sigma$ we have $m_Z^u(A)= \mu^{cu}_{\tilde{Z}} (\pi^{-1}(A)) = \displaystyle \int_{\tilde{Z}} \mu^u (\pi^{-1}(A)) d \mu^{cu}.$ Any element $Z $ is partitioned  $Z = \bigcup Z_i$ where $Z_i$ are atoms of the partition $f^{-1}\sigma$ inside $Z$. Similarly, considering $\tilde{Z}$ as an element of $\mathcal{E}/\xi^{cu}$ we can write $\tilde{Z}= \bigcup \tilde{Z}_i$ where $\tilde{Z}_i$ are atoms of the partition $F^{-1}\xi^{cu}$ inside $\tilde{Z}.$

Let $\phi (a) := - a \log(a)$ which is a  real concave function and remember by Jensen inequality
$$
 \phi (\int g d \alpha) \geq \int (\phi \circ g) d \alpha
$$
for any real valued integrable function $g$ defined on  some measurable space with measure $\alpha.$ The equality holds if and only if $g$ is almost everywhere constant.

 Let $g: \tilde{Z} \rightarrow \mathbb{R}$ be defined as $ g (t) = \mu^u (F^{-1} \xi^u(t))$ and observe that using lemma (\ref{quotient}) and disintegrating $\mu^{cu}$ into $\mu^u$ as conditionals we have:

  $$m^u(Z_i) = \int_{\tilde{Z_i}} \mu^u(F^{-1}\xi^u (t)) d \mu^{cu}_{\tilde{Z}}(t) = \int_{\tilde{Z_i}} g(t) d \mu_{\tilde{Z}}^{cu}(t).$$

  Observe that $$ \int_Z - \log m^u (f^{-1} \sigma (s)) dm^u(s) = \sum_i \int_{Z_i} - \log m^u (f^{-1} \sigma (s)) dm^u(s) $$
$$= - \sum_i m^u(Z_i) \log m^u(Z_i) = \sum_i \phi(m^u(Z_i)) = \sum_i \phi(\int_{\tilde{Z_i}} g d\mu^{cu}_{\tilde{Z}}) $$
On the other hand, 
\begin{equation} 
\begin{split}
\int_{\tilde{Z}} -\log \mu^u_t((F^{-1}\xi^u)(t)) d\mu^{cu}_{\tilde{Z}}(t) & = \sum_i \int_{\tilde{Z}_i}  -\log \mu^u_t((F^{-1}\xi^u)(t)) d\mu^{cu}_{\tilde{Z}}(t) \\
&= \sum_i \int_{\tilde{Z}} \int_{\xi^u(u)} -\chi_{\tilde{Z}_i} \log \mu^u((F^{-1}\xi^u)(t)) d\mu_t^u d\mu_{\tilde{Z}}^{cu} \\
&= \sum_i \int_{\tilde{Z}_i} - \mu^u_t((F^{-1}\xi^u)(t)) \log \mu^u((F^{-1}\xi^u)(t)) d\mu_{\tilde{Z}}^{cu}
\\
&= \sum_i \int_{\tilde{Z}_i} \phi \circ g(t) d\mu_{\tilde{Z}}^{cu}
\end{split}
\end{equation}

So, using Jensen's inequality we obtain the inequality claimed in the theorem.

When $h_{\mu}(F, \xi^u) = h_{m}(f, \sigma),$ we must have equality in Jensen's inequality. Hence, $\pi_* \mu^u_x = m^u_{\pi(x)}$ restricting on the sub-algebra generated by $F^{-1}\xi^u.$ As 
$$
 h_{\mu} (F^n, \xi^u) = n h^u(F, \xi^u) = n h_{m}(f, \sigma) = h_{m}(f^n, \sigma)
$$
applying a similar argument as above, we show that $\pi_*\mu^u_x= m^u_{\pi(x)}$ restricting on the sub-algebra gnerated by $\{F^{-n} \xi^u\}.$ As $\xi^u$ is a generating partition we conclude that $\pi_* \mu^u_x = m^u_{\pi(x)}$. This means that the conditional measures of $\mu$ along the atoms of $\xi^u$ are invariant by the holonomy defined by the fibers and using Lemma \ref{product} the proof of the Theorem (\ref{criterion}) is complete. .
\end{proof}

\subsection{Examples and Application} \label{applicationinv}
As an independent corollary of our criterion for $u-$invariance of disintegration, let us state the following:
Let $M, N$ be compact manifolds, $A : M \rightarrow M$ an Anosov diffeomorphism and $m$ an $A-$invariant probability. Consider the class of partially hyperbolic diffeomorphisms which are skew product over $A:$
$f : M \times N \rightarrow M \times N, (x, \theta) \rightarrow (A(x), f_x(\theta)),$

 Let $f_n$ be skew products as above and $\mu_n$ invariant by $f_n.$ Suppose that $f_n \rightarrow f$ in $C^1-$topology and $\mu_n \rightarrow \mu$ in weak$-*$ topology and $\pi_* \mu_n = m.$ If $\mu_n$ has $u-$invariant disintegration along the fibers $N$, then $\mu$ as well. Indeed, by the theorem \ref{criterion}: 
 $$
h_{\mu_n} (f_n, \mathcal{F}^u(f_n)) = h_m(A).
 $$
 By upper semi-continuity of unstable entropy $$ h_m(A ) =\limsup_{n \rightarrow \infty} h_{\mu_n} (f_n, \mathcal{F}^u(f_n)) \leq h_{\mu}(f, \mathcal{F}^u(f)).$$
 
 As $\pi_* \mu = m$ the above inequality and theorem \ref{criterion} imply $h_{\mu}(f, \mathcal{F}^u(f)) = h_m(A).$ Using again theorem \ref{criterion} we conclude that $\mu$ is $u-$invariant.
\begin{question}
Is it true that the limit of $u-$invariant measures is $u-$invariant in general without assuming that  $\pi_* \mu_n = m$, for a fixed probability $m$?
\end{question}
Now we concentrate on a class of partially hyperbolic dynamics and apply our criterion of invariance principle for measures of maximal  and ``high entropy".
Let $f: M \rightarrow M$ be a partially hyperbolic dynamics satisfying the following conditions:
\begin{itemize}
\item H1. $f$ is dynamically coherent with all center leaves compact,
\item H2. $f$ admits global holonomies, that is, for any $y \in \mathcal{F}^u(x)$ the holonomy map $H^u_{x,y}: \mathcal{F}^c(x) \rightarrow \mathcal{F}^c(y)$ is a homeomorphism. For any $z \in \mathcal{F}^c(x)$, $H^u_{x,y} (z) = \mathcal{F}^u(z) \cap \mathcal{F}^c(y).$
\item H3. $f_c$ is a transitive topological Anosov homeomorphism, where $f_c$ is the induced dynamics satisfying $f_c \circ \pi = \pi \circ f$ and $\pi : M \rightarrow M/\mathcal{F}^c$ is the natural projection to the space of central leaves. In particular there are two foliations $\mathcal{W}^s$ and $\mathcal{W}^u$ which are stable and unstable sets for $f_c.$
\end{itemize}

A large class of partially hyperbolic dynamics denoted by fibered partially hyperbolic systems satisfy (H1) and (H2) and all known examples satisfy (H3). In particular, it is shown in \cite{HP14} that, 
over any 3 dimensional Nil-manifold different from the torus, every partially hyperbolic diffeomorphism satisfies (H1), 
(H2) and (H3).

As a direct consequence of Theorem \ref{criterion} we obtain:

 \begin{proposition} \label{entropyandinvariance}
 Let $f$ be a $C^1$ partially hyperbolic diffeomorphism satisfying H1, H2 and H3 and preserving a probability measure $\mu$ with $\pi_* \mu = m.$ Then $h_{\mu}(f, \mathcal{F}^u) \leq h_m(f_c)$ and equality  holds if and only if $\mu$ admits $u-$ invariant disintegration along the center foliation.
 \end{proposition}

\begin{theorem} \label{jadid}
Let $f$ be a $C^1$ partially hyperbolic diffeomorphism satisfying H1, H2 and H3 with one dimensional center leaves. Let  $\mu$ an $f-$invariant probability with $\pi_* \mu = m.$, then $\mu$  has $u-$invariant disintegration along center foliation if and only if $h_{\mu}(f) = h_{\mu}(\mathcal{F}^u, f).$
\end{theorem}

\begin{proof} 
 If the central foliation is one dimensional then $h_{\mu}(f) = h_{m}(f_c).$ Indeed, by Ledrappier-Walter's variational principle \cite{LW},
\begin{equation} \label{LW}
 \sup_{\hat{\mu}: \pi_ * \hat{\mu} = m} h_{\hat{\mu}} (f) = h_{m}(f_c) + \int_{M/\mathcal{F}^c} h(f, \pi^{-1} (y)) dm(y).
\end{equation}
Since each $\pi^{-1} (y)$ is a circle and its iterates have bounded length we have that $h(f, \pi^{-1} (y)) =0$ that is, fibers do not contribute to the entropy. Hence, by the above equality and the well-known fact that $h_{\mu} (f) \geq h_{m} (f_c)$ we conclude that $h_{\mu}(f) = h_{m}(f_c) = h_m(\mathcal{W}^u, f_c)$. So by Theorem \ref{criterion} we conclude the proof. 
\end{proof}

Now let us apply Theorem \ref{criterion} as an invariance principle when some information about central Lyapunov exponent is available.
\begin{theorem} 
Let $f$ be a $C^{1+\alpha}$ partially hyperbolic diffeomorphism satisfying H1, H2 and H3 and preserving a probability measure $\mu.$ If all central Lyapunov exponents of $f$ are non-positive, then $\mu$ admits  $u-$invariant disintegration along central foliation.
\end{theorem}
\begin{proof}
Let $\pi: M \rightarrow M/\mathcal{F}^c$ be the natural quotient map and $m = \pi_* \mu$ the projected invariant measure of $f_c.$ By $u-$invariance of the disintegration along center fibers we mean the following: There is a full $m-$measurable subset of $M/\mathcal{F}^c$ such that for any two points $x, y$ in the same unstable set $y \in \mathcal{W}^u(x)$ we have $\mu_y = (H^u_{x,y})_* \mu_x.$
The proof is a direct application of Theorem \ref{criterion}  and Ledrappier-Young\textquotesingle s  result: 
$$
h_{\mu}(f) = h_{\mu}(f, \mathcal{F}^u) \leq h_{m}(f_c, \mathcal{W}^u) \leq h_{m} (f_c)
$$
As $h_m(f_c) \leq h_{\mu}(f)$ we obtain equality in the above inequality and by theorem \ref{criterion} a we conclude  that the disintegration of $\mu$ along central leaves is $u-$invariant.
\end{proof}

Using the Hopf argument and the above theorem we can conclude the following version of invariance principle, as it is done in Avila-Viana's paper.
\begin{corollary}  \cite[Proposition~4.8]{AV}  \label{continuous}
Let $f$ be as in the above theorem and all central exponents vanish. Moreover, suppose that $m= \pi_*(\mu)$ has local product structure. Then, $x \rightarrow \mu_x$ is a continuous family.
\end{corollary}

The above version of invariance principle can be applied to characterize all measures of maximal entropy of accessible dynamically coherent partially hyperbolic diffeomorphisms with one dimensional compact center. 

\begin{theorem} \label{dichotomy} \cite{RHRHTU}
Let $f : M \rightarrow M$ be a $C^{1+\alpha}$ partially hyperbolic diffeomorphism of a 3-dimensional closed manifold $M$. Assume that $f$ is dynamically coherent with compact one dimensional central leaves and has the accessibility property. Then $f$ has finitely many ergodic measures of maximal entropy. There are two possibilities:
\begin{enumerate}
\item  (rotation type) $f$ has a unique entropy maximizing measure $\mu$. The central Lyapunov exponent $\lambda_c(\mu)$ vanishes and $(f, \mu)$ is isomorphic to a Bernoulli shift,

\item (generic case) $f$ has more than one ergodic entropy maximizing measure, all of which with non vanishing central Lyapunov exponent. The central Lyapunov exponent $\lambda_c(\mu)$ is nonzero and $(f, \mu)$ is a finite extension of a Bernoulli shift for any such measure $\mu.$ Some of these measures have positive central exponent and some have negative central exponent. \end{enumerate}
Moreover, the diffeomorphisms fulfilling the conditions of the second item form a $C^1-$open and $C^{\infty}-$dense subset of the dynamically coherent partially hyperbolic diffeomorphisms with compact one dimensional central leaves.
\end{theorem}

 \subsubsection{Maximal entropy measures which are not $u-$maximal measure} \label{mmenotu}
In what follows, we give an example of measure of maximal entropy which is not measure of maximal $u-$entropy. Let $f$ satisfy the hypothesis of Theorem \ref{dichotomy}. 
First of all observe that as $\mathcal{F}^c$ is one dimensional then $h_{top}(f) = h_{top}(\mathcal{F}^u).$ (See \cite{HHW} or \cite{WZ}.) 

Let $\mu$ an ergodic measure of maximal entropy such that $\lambda^c(\mu) > 0.$ By Theorem~ \ref{dichotomy} $\mu$ has $s-$invariant center disintegration and   the conditional measures are not $u-$invariant. Indeed, if $\mu$ is maximal entropy and has both $s$ and $u-$invariant conditional measures, then the proof of theorem \ref{dichotomy} shows that $\lambda^c (\mu) =0.$

So by Proposition \ref{jadid}  we conclude the following strict inequality $$h_{\mu}(f, \mathcal{F}^u) < h_m(f_c) \leq h_{top} (f_c) = h_{top}(f) = h_{top}(\mathcal{F}^u).$$
Both equalities above come from one dimensionality and compactness of the center foliation.

So we conclude that $\mu$ is not a $u-$maximal measure.



\subsubsection{High entropy measures in one-dimensional center case} \label{highentropy}

In this section study the Lyapunov exponent of high entropy measures. Let $f$ be a partially hyperbolic diffeomorphism with finitely many ergodic entropy maximizing measures. Suppose that all these measures are hyperbolic. We would like to know whether high entropy measures inherit the hyperbolicity of maximal entropy measures or not.

 Let us pose a similar question: $f$ partially hyperbolic with finitely many measures of maximal entropy, suppose that $\mu_n$ is a sequence of measures with $h_{\mu_n} \rightarrow h_{top}(f)$ and $\lambda^c(\mu_n) =0.$ Is it true that $f$ should admit a non-hyperbolic maximizing measure?

In the following theorem we prove a satisfactory answer in the context of partially hyperbolic diffeomorphisms with compact one dimensional center.

\begin{theorem}
Let $f$ be as in theorem \ref{dichotomy} without ergodic non-hyperbolic measure of maximal entropy (the generic case), then high entropy measures are hyperbolic, that is, there exist $\alpha, \beta > 0$ such that for any ergodic invariant measure $\mu$ if $h_{\mu} \geq h_{top}(f) - \alpha$ then $|\lambda^c(\mu)| \geq \beta.$
\end{theorem}
\begin{proof}
Suppose by contradiction that there is a sequence of ergodic measures $\mu_n$ such that $h_{\mu_n} \rightarrow h_{top}(f)$ and $|\lambda^c(\mu_n)| \rightarrow 0.$ Without loss of generality we may suppose that $\mu_n \rightarrow \mu$ and $\lambda^c(\mu_n) \leq 0.$ By Ledrappier-Young result we have $h_{\mu_n}(\mathcal{F}^u, f) = h_{\mu_n}(f) \rightarrow h_{top}(f).$ So by upper semi-continuity of unstable entropy we obtain that $h_{\mu}(\mathcal{F}^u, f) = h_{top}(f)$ and consequently $h_{\mu}(\mathcal{F}^u, f) = h_{\mu}(f).$ By theorem \ref{jadid} we conclude that $\mu$ has $u-$invariant disintegration along central foliation.  As $\mu$ is a measure of maximal entropy, it has an ergodic decomposition into finitely many ergodic components among ergodic measures of maximal entropy of $f$. We claim that no ergodic component of $\mu$ can have positive central Lyapunov exponent. Indeed, if by contradiction there exists $\mu^{+}$ with positive central exponent as an ergodic component of $\mu$ then $\mu^+$ is has both $s$ and $u$ invariant disintegration. So, this implies that $f$ is of rotation type which contradicts our hypothesis about $f$.
\end{proof}

\begin{question} \label{q.vp}
In the above theorem, we do not know what is the value of following supremum: $$\sup \{h_{\mu}(f); \lambda^c(\mu)=0\} $$
\end{question}

We observe that in \cite{DGR} the authors prove a variational principle on the level sets of  {\it non vanishing} central Lyapunov exponent. However it is not  clear whether the above supremum is equal to the topological entropy of the set of points with zero central Lyapunov exponent or not.

\subsubsection{Non-invertible Bochi-Ma\~{n}\'{e} result} \label{bochi}

Suppose $f: M \rightarrow M$ is a continuous uniformly expanding map of a compact manifold. That is, there are $\rho > 0, \sigma > 1$ such that for every $x \in M:$
\begin{enumerate}
\item $d(f(x), f(y)) \geq \sigma d(x, y)$ for every $x, y \in B(x, \rho)$ and
\item $f(B(x, \rho))$ contains the closure of $B(f(x), \rho).$
\end{enumerate}
Let $\mu$ be an ergodic $f$-invariant probability with full support.
One can construct the inverse limit space $\hat{M}.$ Recall that $\hat{M}$ is the space of sequences $(x_{-n})_{n}$ such that $f(x_{-n})= x_{-n+1}$ and 
$$
\hat{f}(\cdots, x_{-1}, x_0) = (\cdots, x_0, f(x_0)).
$$ 
Let $\pi: \hat{M} \rightarrow M$ be the projection $\pi((x_{-n})_{n}) = x_0.$ Then $\hat{f}$ is a hyperbolic homeomorphism (topologically Anosov) and satisfies $\pi \circ \hat{f} = f \circ \pi.$ There exists a lift of $\mu$ invariant by $\hat{f}$ which is denoted by $\hat{\mu}$ and one important property of $\hat{\mu}$ is that it has local product structure and consequently we are  in the setting of Corollary \ref{continuous}. 
However, we need to make precise the notion of $u-$invariance and $s-$invariance of measures by defining the holonomy maps.

As before we define the projective cocycle corresponding to  a continuous map $A: M \rightarrow SL(2)$ and  denote it by $F_A: M \times \mathbb{P}\mathbb{R}^2 \rightarrow M \times \mathbb{P}\mathbb{R}^2.$ 

Let $\hat{A} : \hat{M} \rightarrow SL(2)$ be defined by $\hat{A} = A \circ \pi.$ In this way we define a map $\hat{F}_{\hat{A}}\ : \hat{M} \times \mathbb{P}\mathbb{R}^2 \rightarrow \hat{M} \times \mathbb{P}\mathbb{R}^2.$    As $\hat{A}$ is constant on local stable sets, $\hat{F}_{\hat{A}}$ admits trivial local stable holonomy. That is 
$$  H^s_{\hat{x}, \hat{y}} = id  \quad \text{for any} \quad \hat{y} \in W^s_{loc}(\hat{x}), \hat{x}, \hat{y} \in \hat{M}.$$
If $A$ satisfies the so called $u-$bunched condition, then one can define local unstable holonomies between fibers too. We say that $A$ is $u-$bunched if there is $\theta > 0$ such that $A$ is $\theta-$H\"{o}lder and 
$$
 \|A(x)\| \|A(x)^{-1}\| \sigma^{-\theta} < 1.
$$

The main point of $u-$invariance principle is that if $A$ is $u-$bunched and $\hat{F}_{\hat{A}}$ has no $u-$invariant measure then $\lambda(A) > 0.$ Indeed, if not then $\lambda(A) = 0$ (recall that $\lambda(A) \geq 0$) and any $\hat{F}_{\hat{A}}-$invariant measure $\hat{m}$ which projects down on $\hat{\mu}$ is $u-$invariant. In fact any such measure is also $s-$invariant and then using Hopf type argument (Corollary \ref{continuous}) its disintegration along fibers is defined continuously for all point $\hat{x} \in \hat{M}$ and is both $s$ and $u-$invariant. An invariant measure for $\hat{F}_{\hat{A}}$ which is both $s$ and $u-$invariant is called $su-$state and in what follows, we write $A$ has $su-$state.

Besides the  above conclusion, the authors in \cite{VY} prove continuity of Lyapunov exponent. More precisely:

\begin{theorem} \cite{VY} \label{cont}
If $A$ is $u-$bunched and has no $su-$state then $\lambda(A)> 0 $ and $A$ is a continuity point for the function $B \rightarrow \lambda(B)$ in the space of continuous maps $B: M \rightarrow SL(2)$ equipped with the $C^0-$topology. In particular, the Lyapunov exponent $\lambda(.)$ is bounded away from zero in a $C^0-$neighborhood of $A.$
 \end{theorem}

Now let $M = \mathbb{R}/\mathbb{Z}$, $f: M \rightarrow M, f(x) = kx$ mod $\mathbb{Z}$ and $\mu$ the Lebesgue measure. The inverse limit of $f$ is the well-known solenoid ($k-$folded) and we denote by $\hat{f}$. Let $A(x) = A_0R_x$  where $A_0$ is a hyperbolic matrix in the special group $SL(2)$ and $R_x$ the rotation by angle $2 \pi x.$ For large $k$ the cocycle defined by $f$ and $A$ is $u-$fiber bunched. 

One can prove (using an originally Herman's idea, see \cite{VY}) that for large enough $k,$ $A$ has no $su-$state  and consequently $\lambda(A) > 0.$ Using continuity in $C^0$ topology (Theorem \ref{cont}) we conclude that $ \lambda(B) > 0$ for any $B : M \rightarrow SL(2)$ near to $A.$ 

As the cocyle defined by $A$ is not uniformly hyperbolic (using the same kind of idea from Herman) Viana and Yang got a counter-example to Bochi-Mañé type result when the base dynamics is non-invertible. Bochi \cite{Bochi} proved that every continuous $SL(2)-$cocycle over an aperiodic invertible system can be approximated in $C^{0}-$topology by cocycles whose Lyapunov exponents vanish, unless it is uniformly hyperbolic. These results was extended by Bochi and Viana \cite{BoV} to higher dimension. 
It is interesting to understand the similar results or counter-examples as above in higher dimensional case.

  The paper of Viana-Yang deals with non $u-$bunched cocycles and using a similar argument to the entropy criterion (Theorem \ref{criterion}) they obtain the following theorem (Theorem \ref{nobunch}).  

Let $f$ be a $C^{1+\epsilon}$ expanding map on a compact manifold $M$ and $A: M \rightarrow SL(2)$ be a $C^{1+\epsilon}$ function. All the objects $\mu, F, \pi, \hat{M}, \hat{f}, \hat{\mu}, \hat{A}$ and $\hat{F}$ are as above. An invariant section is a continuous map $\hat{\xi} : \hat{M} \rightarrow \mathbb{P}\mathbb{R}^2$ or $\hat{\xi}: \hat{M} \rightarrow \mathbb{P}\mathbb{R}^{2,2}$ (the space of pair of points in $\mathbb{P}\mathbb{R}^2$) such that 
$$ \hat{A} (\hat{x}) \hat{\xi} (\hat{x}) = \hat{\xi}(\hat{f}(\hat{x})) \quad \text{for every} \quad \hat{x} \in \hat{M}.$$ 

\begin{theorem} \label{nobunch} \cite{VY}
If $A$ admits no invariant section in $\mathbb{P}\mathbb{R}^2$ or $\mathbb{P}\mathbb{R}^{2,2}$ then it is a continuity point of the map $B \rightarrow \lambda(B)$ in the space of continuous maps $B: M \rightarrow SL(2)$ equipped with the $C^0$ topology. Moreover, $\lambda(A) >0$ if and only if there exists some periodic point $p \in M$ such that $A^{per(p)}(p)$ is a hyperbolic matrix. In that case, $\lambda(.)$ is bounded from zero for all continuous cocycles in a $C^0-$neighborhood of $A.$
\end{theorem}

\subsection{Invariance principle for partially hyperbolic diffeomorphisms without compact center leaves} \label{cropol}

In this subsection we announce an invariance principle which deals with partially hyperbolic diffeomorphisms acting quasi-isometrically on center leaves. An advantage of this invariance principle is its applicability for the partially hyperbolic dynamics without compact center leaves (compared with Tahzibi-Yang \cite{TY} result) and non-smooth invariant measures (compared with Avila-Santamaria-Viana-Wilkinson \cite{ASVW} result).  An important class of partially hyperbolic diffeomorphisms namely Anosov discretized flows (See \cite{BFFP} for definition.) are included in this class. 

\begin{definition}
A partially hyperbolic and dynamically coherent diffeomorphism $f$ is quasi-isometric in the center if there exists $K_0 \geq 1$ and $c_0 > 0$ such that for every $x, y \in M$ satisfying $\mathcal{F}^c(x)= \mathcal{F}^c(y)$ and every $n \in \mathbb{Z}:$
$$
K_0^{-1} d^c(x, y) - c_0 \leq d^c(f^n(x), f^n(y)) \leq K_0 d^c(x, y) + c_0.
$$
\end{definition}

\begin{proposition} \label{holocp}
Let $f$ be a partially hyperbolic, dynamically coherent quasi isometric in the center, $C^1$-diffeomorphism and $\mu$ be an invariant measure. Then there exists a full set $X \subset M$ such that for any $x, y \in X$ with $\mathcal{F}^{cu}(x)= \mathcal{F}^{cu}(y)$ the following holds:
For any $z \in \mathcal{F}^c(x),$ the leaves $\mathcal{F}^u(z), \mathcal{F}^c(y)$ intersect at a unique point, denoted $H^u_{x,y}(z).$ The global holonomy map $H^u_{x,y}: \mathcal{F}^c(x) \rightarrow \mathcal{F}^c(y)$ is a homeomorphism. 
\end{proposition}

We say that the center disintegration $\{\mu^c_x\}_{x \in M}$ of $\mu$ is invariant under
unstable holonomies  if for $\mu$-almost every $x, y$ satisfying
$\mathcal{F}^{cu}(x) = \mathcal{F}^{cu}(y)$ there exists $K_{x,y} > 0$ such that $(H^u_{x,y})_* \mu^c_x = K_{x,y} \mu^c_y$
where the holonomy map $h^u_{x,y}$
 is uniquely defined by proposition \ref{holocp}.

Now we are ready to announce the main invariance principle result due to Crovisier-Poletti. The first theorem is related to the first part of invariance principle results ($u-$invariance) and the second one deals with continuity of disintegration.

\begin{theorem}
Let $f$ be a partially hyperbolic, dynamically coherent, quasi-isometric in the center, $C^1$-diffeomorphism and let $\mu$ be an ergodic measure. If $h(f,\mu) = h_{\mu}(f,\mathcal{F}^u)$, then the center disintegration $\{\mu^c_x\}_{x \in M}$
 is $u$-invariant.
\end{theorem}

\begin{theorem}
Let $f$ be a partially hyperbolic, dynamically coherent, quasi-
isometric in the center, $C^1-$diffeomorphism and let $\mu$ be an ergodic measure. If $h_{\mu}(f^{-1}, \mathcal{F}^s) = h_{\mu}(f, \mathcal{F}^u) = h_{\mu}(f)$ and if $\mu$ has local $cs \times u$-product structure,
then there exists a family of local center measures $\{\nu^c_x\}_{x \in \supp(\mu)}$ on the sup-
port of $\mu$ which is continuous, $f$-invariant, $s$-invariant, $u$-invariant and
extends the center disintegration of $\mu$.
\end{theorem} 

\section{Margulis family and $u-$maximizing measures } \label{s.margulis}

It is known that a mixing uniformly hyperbolic diffeomorphism $f: X \rightarrow X$  admits a unique measure of maximal entropy (Bowen-Margulis measure) $\mu$. Following the construction of $\mu$ by Margulis, we have that conditional measures of $\mu$ along $\cF^s$ and $\cF^u$ constitute of two ``Margulis families" $\{m^s_x\}, \{m^u_x\}, x \in X$ which are respectively contracted and expanded by the dynamics. In some sense Margulis family of measures on global leaves are ``making the action of $f$ linear." 
\begin{definition}\label{def-sys-meas}
Given a foliation $\cF$ of some connected manifold $M$  a system of measures on leaves  is a family $\{m_x\}_{x\in M}$ such that:
 \begin{enumerate}
  \item for all $x\in M$, $m_x$ is a non-trivial  Radon measure on $\cF(x)$;
  \item for all $x,y\in M$, $m_x=m_y$ if $\cF(x)=\cF(y)$;
 
\end{enumerate}
We say the system is continuous (measurable) if
 \item $M$ is covered by foliation charts $B$ such that $x\mapsto m_x(\phi|\cF_B(x))$ is continuous (measurable), i.e.  for any $\phi\in C_c(B)$, where $\mathcal{F}_B(x)$ is the connected component of $\mathcal{F}(x) \cap B$ containing $x$.
\end{definition}

The Radon property (i) means that each $m_x$ is a Borel measure and is finite on compact subsets of the leaf $\cF(x)$ (here, and elsewhere, we consider the intrinsic topology {on each leaf}).

\begin{definition} \label{def.margulisfamily}
Assume that $\cF$ is a foliation which is invariant under some diffeomorphism $f:M\to M$, i.e., for all $x\in M$, $f(\cF(x))=\cF(f(x))$. 
A  system of measures $\{m_x\}_{x\in M}$  is called a Margulis family if there is a  constant  $D>0$ such that for all {$x\in M$}:
  \begin{equation}\label{eq-dilation}
 { f_*m_x=D^{-1}m_{f (x).} }
  \end{equation}
 $D$ is called the \new{dilation factor}. We call the family $\{m_x\}_{x\in M}$ a \new{Margulis system} on $\cF$ with dilation factor $D$ and the measures $m_x$ the \new{Margulis $\cF$-conditionals}. 
\end{definition} 

We say a Margulis family is continuous (measurable) if the system is continuous (measurable).
The continuity property seems a strong condition. However, the many examples of Margulis measures are  invariant by the holonomy of a transverse foliation and consequently they are continuous.

Let $\mathcal{F}$ and $\mathcal{G}$ be two transverse  invariant foliations of a compact manifold $M$. Let $\{m_x\}$ be a Margulis family defined on the leaves of $\mathcal{F}$. For any two points $x, y$ such that $y \in \mathcal{G}(x)$, by transversality of $\mathcal{F}$ and $\mathcal{G}$ we may define a (local) holonomy map $H_{x,y}$  from a ball $B_{\mathcal{F}}(x)$ to $B_{\mathcal{F}}(y)$ which is a homeomorphism to its image. We say $\{m_x\}$ is holonomy invariant if $m_y (H_{x,y}(Z)) = m_x(Z) $ for any $Z \subset B_{\mathcal{F}}(x).$

For a general invariant foliation, it is not clear whether $D$ is unique or not. That is, there may exist two different Margulis families with different dilation constants. However, see proposition \ref{jinquestion} where it is shown that in the case of unstable foliation of a partially hyperbolic diffeomorphism, the dilation factor should be equal to the unstable topological entropy.

\begin{proposition}\label{prop-Ledrappier}   
Let $f\in\Diff^2(M)$ be a partially hyperbolic diffeomorphism with a measurable Margulis $\mathcal{F}^u$-system $\{m^u_x\}_{x\in M}$ ($\mathcal{F}^u$ stands for the unstable foliation) with dilation factor $D_u>0$. Suppose that $m^u_x$ is fully supported on $\mathcal{F}^u (x).$ Then for any invariant measure $\mu,$
$$h_{\mu}(f, \mathcal{F}^u)\leq\log D_u.$$
Moreover, the equality holds if and only if the disintegration of $\mu$ along $\Fu$ is given by $m^u_x$ (after normalization by a constant), $\mu$-a.e.

\end{proposition}
Using  Ledrappier-Young result (Theorem \ref{expdiment}),
\begin{corollary}
 In the above proposition,
if $\mu\in \mathcal{M}_e(f)$ (an ergodic invariant measure) with all  central Lyapunov exponents non-positive, then $$h(f,\mu)\leq\log D_u.$$ Moreover, the equality holds if and only if the disintegration of $\mu$ along $\Fu$ is given by $m^u_x$ (after normalization by a constant), $\mu$-a.e.
\end{corollary}

\begin{proof}
The proof with more details is in \cite{BFT} and we repeat here because of its independent interest. The ideas come from Ledrappier and use a convexity argument.

Let $\xi$ be an  measurable increasing partition subordinated to $\mathcal{F}^u.$
Define a normalized family of measures adapted to the partition $\xi:$
$$
 m_{\xi(x)} (A) := \frac{m_x^u (A \cap \xi(x))}{ m^u_x(\xi(x))}.
$$ Observe that above ratio is well defined, as $m_x^u$ is fully supported and $\xi(x)$ contains an open set.
The dilation property of Margulis measures yields: 
$$
m_{\xi(x)}((f^{-1}\xi)(x)) = D_u^{-1} \frac{m^u_{f(x)} (\xi( fx)) }{m^u_x(\xi(x))}.
$$ 
{Following Ledrappier, observe that $g(x):=-\log m_{\xi(x)}((f^{-1}\xi)(x)) \geq0$ so that, by the pointwise ergodic theorem, we know $\lim_n \frac1nS_ng(x)$ (possibly $+\infty$) exists almost everywhere. To identify this limit,  observe that it is also a limit in probability. Taking the logarithm of the previous identity, we see that $g(x)=h(fx)-h(x)+\log D_u$ for a measurable function $h$. Therefore the limit in probability, and consequently almost everywhere, is the constant  $\log D_u$. Thus $g$ is integrable with:}
 $$
- \int \log m_{\xi(x)}((f^{-1} \xi)( x)) d \mu =  \log D_u.
 $$
Now recall that $ h(f,\mu,\mathcal{F}^u) = - \displaystyle\int \log\mu_{\xi(x)}((f^{-1}\xi) (x))\, d\mu $ and so,
$$
- \int \log\mu_{\xi(x)}((f^{-1}\xi)(x))\, d\mu \leq - \int \log m_{\xi(x)}((f^{-1}\xi)(x)) d\mu = \log(D_u).
$$

The inequality comes from Jensen's inequality and the (strict) concavity of the logarithm.
The case of equality for Jensen's inequality yields that this is an equality if and only if $\mu_{\xi(x)}((f^{-1}\xi)(x)) = m_{\xi(x)}((f^{-1}\xi)(x))$ for $\mu$-a.e. $x\in M$. Replacing $\xi$ by $f^{-n}\xi$, we obtain that
 $$
      \mu_{f^{-n}\xi(x)}(f^{-n-1}\xi)(x)) = \frac{m^u_x((f^{-n-1}\xi)(x))}{m^u_x((f^{-n}\xi)(x))}
  $$
so
  $$\begin{aligned}
    \mu_{\xi(x)}((f^{-n-1}\xi)(x)) &= \prod_{k=0}^n \frac{\mu_{\xi(x)}((f^{-k-1}\xi)(x))}{\mu_{\xi(x)}((f^{-k}\xi)(x))}
    =
    \prod_{k=0}^n \mu_{(f^{-k}\xi)(x)}((f^{-k-1}\xi)(x))\\
     &
     =\prod_{k=0}^n \frac{m^u_x((f^{-k-1}\xi)(x))}{m^u_x((f^{-k}\xi)(x))}    
    =\frac{m^u_x((f^{-n-1}\xi)(x))}{m^u_x(\xi(x))}.
  \end{aligned}$$
Since $\xi$ is generating and increasing,  the disintegration of $\mu$ along $\xi$ is given by the Margulis $u$-conditionals as claimed.
\end{proof}





\subsection{Construction of maximal entropy measures using Margulis family} \label{margulishyperbolic}

The Proposition \ref{prop-Ledrappier} is central to obtain measures of maximal entropy in the uniformly hyperbolic case and beyond. For simplicity, let us construct the measure of maximal entropy for a transitive Anosov diffeomorphism. 

Margulis constructed measure of maximal entropy for Anosov flows with minimal strong stable and minimal strong unstable foliations. We review this method for Anosov diffeomorphisms.  For partially hyperbolic diffeomorphisms close to Anosov flows, see \cite{BFT}.

The very first result to prove is the existence of Margulis family along unstable and stable foliations. 

\begin{theorem}
If $f: M \rightarrow M$ is a $C^{1}$ Anosov diffeomorphism with minimal stable and unstable foliations, then there are two Margulis families of measures $\{m^u_x\}_{x \in M}, \{m^s_{x}\}_{x \in M}$ with dilation factor $D^u, D^s$ such that 
\begin{itemize}
\item dilation factors $D^u, D^s = \frac{1}{D^u} = e^{h_{top(f)}},$ 
\item $\{m^u_x\}$ is invariant by $\mathcal{F}^s-$holonomy.
\item $\{m^s_x\}$ is invariant under $\mathcal{F}^u-$holonomy.
\end{itemize}  
Moreover, one can construct $m$ a measure of maximal entropy which has local product structure.
\end{theorem}

The construction of Margulis families claimed in the above theorem, using Margulis original method,  needs a higher regularity than $C^1$ (because of some absolute continuity result). However in the Anosov setting  we approximate $f$ with a $C^2$ Anosov diffeomorphism $g$. Then, we use structural stability to push the Margulis family of measures constructed for $g$ to obtain a family with the same dilation factors. Indeed, topological conjugacy between $f$ and $g$ preserves stable and unstable foliations and sends  Margulis family to  Margulis family. We thank an anonymous referee to point out this to us.

\begin{remark} We recall a result of Ruelle-Sullivan \cite{RS} where they identified measure of maximal entropy for Axiom-A diffeomorphisms as product of measures with dilation property like in Margulis family. However, they used symbolic dynamics and 
used the product nature of maximal entropy measure for symbolic dynamics.
\end{remark}

Now, let us comment on the construction of the invariant measure of maximal entropy using the Margulis families $\{m^u\}$ and $\{m^s\}$. We give a local description of this measure in a small rectangle $\mathcal{R} \subset M$. By a rectangle we mean $\mathcal{R}$ such that for any two points $x, y \in \mathcal{R}, \mathcal{F}^s(\mathcal{R}, x) \cap \mathcal{F}^u(\mathcal{R}, y) = \{z\} \in \mathcal{R},$ where $ \mathcal{F}^s(\mathcal{R}, x) $ denotes the connected component of $\mathcal{F}^s(x) \cap \mathcal{R}$ and containing $x.$ 
Fix any $x \in \mathcal{R}$ and define $$\mu|{\mathcal{R}} = \displaystyle \int_{\mathcal{F}^u(\mathcal{R}, x)} m^s_y d m^u(y).$$ Observe that by invariance of $m^u$ under the stable holonomy, the measure defined in this rectangle is independent of the choice of $x \in \mathcal{R}.$ Indeed, $\mu | \mathcal{R}$ has product structure. As $M$ is compact we can define a measure $\mu$ on $M$ by defining it locally. Observe that $\mu | \mathcal{R}$ is a finite measure, by the fact that $m^s$ and $m^u$ are Radon measures (finite on compact sets) and so globally $\mu$ is a finite measure.

Now, observe that $m(f(A)) = D^u D^s m(A)$ for any $A \in \mathcal{R}$ and consequently for any Borel subset $A.$ In particular, as $f(M)= M$ and $ 0 \neq \mu(M) < \infty,$ we have that 
$ D^u D^s = 1$ and immediately we conclude that $\mu$ is an invariant measure.

Finally we prove that $\mu$ is a measure of maximal entropy. Indeed, by Proposition \ref{prop-Ledrappier}, for any invariant measure $\nu:$
$$
 h(f, \nu ) \leq \log(D^u) \quad \text{and} \quad h(f, \mu) = \log(D^u).
$$

\subsubsection{Partially hyperbolic setting} \label{margulisph}
In \cite{BFT} we pushed the arguments of Margulis  in the partially hyperbolic setting:

\begin{proposition} \label{cumargulis}
Let $f \in \Diff^{1+\alpha}(M)$ on  a compact manifold $M$ with a  dominated splitting $E^s \oplus E^{cu}$ with $E^s$ uniformly contracted.  Assume that:
 \begin{enumerate}
 \item there are foliations $\cF^{cu}$ and $\cF^s$ which are tangent to, \ $E^{cu}$ and $E^s$ respectively;
 \item $\cF^s$ is minimal.
 \end{enumerate} 
 Then there is a Margulis system $\{m_x^{cu}\}_{x \in M}$ on $\mathcal{F}^{cu}$ which is invariant under $\Fs$-holonomy and such that each $m^{cu}_x$ is atom-less, Radon, and fully supported on $\cF^{cu}(x)$.
\end{proposition}

\begin{remark}
The minimality condition in the above proposition is not necessary as it will appear in a joint work with Buzzi, Crovisier and Poletti.
\end{remark}

In general given a partially hyperbolic diffeomorphism it is not clear whether or not there exists a family of Margulis measures $\{m^u\}_{x \in M}.$ However,

\begin{proposition} \label{jinquestion}
If there exists a family of continuous Margulis measures $\{m^u_x\}_{x \in M}$ fully supported on $\mathcal{F}^u(x)$ with dilation factor $D^u$ then $h_{top}^u(f) = \log D^u.$ 
\end{proposition}  
\begin{proof}
By the proof of Proposition \ref{prop-Ledrappier} we conclude that for any invariant measure $\mu$, $h_{\mu}(\mathcal{F}^u, f) \leq \log(D^u).$ So by variational principle along unstable foliation we get $h^u_{top}(f) \leq \log(D^u).$ Now, we construct an invariant measure with unstable entropy  equal to $\log(D^u).$ Indeed, let $\{m^u_x\}_{x \in M}$ be the Margulis family as in the hypothesis. Let $\sigma$ be an open disk in $\mathcal{F}^u(x)$ for some $x \in M$ with $m^u(\sigma) > 0.$ We can take $\sigma$ in such a way that its bounday has zero $m^u-$measure. 

Take $\mu$  any accumulation point of $\{\frac{1}{n} \sum_{i=0}^{n-1} f^i_* m^u_{\sigma}\}$ where $m^u_{\sigma}$ is the normalization of $m^u_x$ supported on $\sigma.$ 

 Similar to the argument of Pesin-Sinai for the construction of $u-$Gibbs measures we conclude that $\mu$ has disintegration along unstable foliation equivalent to the Margulis family $\{m^u_x\}.$ In fact using Margulis dilation property, the argument here is simpler.
 
 Let $B$ be a foliated box for $\mathcal{F}^u.$ For any $x \in B$ denote by $\mathcal{F}^u(x, B)$ as the plaque of the foliation passing through $x$ inside $B.$ Let $\mu_n := \frac{1}{n} \sum_{i=0}^{n-1} f^i_* m^u_{\sigma}.$ In the following lemma, we claim that the Rokhlin disintegration of the normalized restriction of $\mu_n$ on $B$ along plaques $\mathcal{F}^u(x, B)$ coincides a.e with the normalized Margulis measure, i.e $\frac{1}{m^u(\mathcal{F}^u(x, B))}m^u_x.$
 
\begin{lemma}
Let $\{m^u_x\}$ be a continuous Margulis family with dilation $D^u$ and $\sigma$ a $u-$disc with $m^u(\sigma) > 0$ such that $m^u(\partial \sigma) =0$ where $\partial \sigma$ stands for the boundary of $\sigma$ inside $\mathcal{F}^u$ leaf. Any accumulation point of $\{\frac{1}{n} \sum_{i=0}^{n-1} f^i_* m^u_{\sigma}\}$ has conditional measure along $\mathcal{F}^u$ coinciding with $m^u.$
\end{lemma}

\begin{proof}

For the sake of simplicity suppose that $m_x^u (\sigma) =1;$ i.e, the restriction of $m^u_x$ on $\sigma$ (disc around $x$) coincides with its normalization.   Observe that by dilation property,
$$
 \mu_n = \frac{1}{n} (m^u_{\sigma} + D^{-1}m^u_{f(\sigma)}+ \cdots+ D^{-n} m^u_{f^n(\sigma)}),
$$
where $m^u_{f^j(\sigma)}$ is the restriction  of $m^u_{f^j(x)}$ on $f^j(\sigma).$ Observe that $D^{-j}m^u_{f^j(\sigma)}$ is a probability measure.

Let $f^j(\sigma) \cap B = G_j \cup B_j$ where $G_j = \bigcup G_{j, i}$ is the union of (complete) plaques insides $B \cap f^j(\sigma)$ and $B_j$ are union of proper subsets of plaques in $B.$

So, $\mu_{n, B}$ the restriction of $\mu_n$ on $B$ satisfies the following: $$\mu_{n,B} = \frac{1}{n} \sum_{j=1}^{n} D^{-j} (m^u_{G_j}  + m^u_{B_j} ).$$ 

Let just consider the measure on complete discs: 
 $$
\nu_n := \frac{1}{n} \sum_{j=1}^{n} D^{-j} m^u_{G_j}.
$$
Observe that $f^{-n}(B_n)$ is inside an exponentially small  neighbourhood of the boundary of $\sigma.$ As $m^u(\partial \sigma) =0,$  the measure of $\epsilon-$neighbourhood of $ \partial \sigma$ tends to zero when $\epsilon$ goes to zero which implies  $$\lim_{n \rightarrow \infty} D^{-n} m^u (B_n) = \lim_{n \rightarrow \infty} m^u(f^{-n} (B_n))= 0.$$

So, the accumulation points of $\mu_{n, B}$ and $\nu_n$ are the same. 
Now we analyse the disintegration of $\frac{1}{\nu_n(B)} \nu_n$ (normalization of $\nu_n$ on $B$) along the plaque partition inside $B.$ We show that on each $G_{j, i}$ the conditional probability measure is $\frac{1}{m^u_{f^j(x)}(G_{j, i})}m^u_{G_{j, i}}$ which is exactly the normalization of Margulis measure. To this end we identify the quotient measure (quotient with respect to the palques partition): 
$$
 \frac{1}{n \nu_n(B)} \sum_j D^{-j} m^u_{G_j}  = \frac{1}{n \nu_n(B)} (\sum_j  D^{-j}\sum_i m^u_{G_{j,i}}  )
$$
$$
= \frac{1}{n \nu_n(B)} (\sum_j \sum_i  D^{-j} m^u_{f^j(x)}(G_{j, i})  \frac{m^u_{G_{j,i}}}{m^u_{f^j(x)}(G_{j, i})})
$$
So, it is enough to consider the quotient measure on the space of plaques by $$\sum_j\sum_i \frac{D^{-j}m^u_{f^j(x)}(G_{j, i})}{n \nu_n(B)}  \delta_{j, i}$$ where $\delta_{j, i}$ represents the dirac measure on the element of the quotient space corresponding to $G_{j, i}.$ 

Observe that up to now, we have proved that the conditional measures of $\nu_n$ coincide with Margulis measures (understood after normalization).  Fix a foliated box $B$ and a continuous function $\phi$. Indeed, Without loss of generality suppose that $\nu_n \rightarrow \nu.$ Then
$$\begin{aligned}
 \int_B \phi d\nu = \lim_{n \rightarrow \infty}\int_B \phi d \nu_n &= \lim_{n \rightarrow \infty}  \int_B \int_{\mathcal{F}^u(B, x)} \frac{1}{m^u_x(\mathcal{F}^u)(B, x)}\phi dm_x^u d \nu_n  \\ 
 &=  \int_ B \int_{\mathcal{F}^u(B, x)} \frac{1}{m^u_x(\mathcal{F}^u)(B, x)} \phi dm_x^u d \nu 
\end{aligned} $$
The last equality comes from the continuity of Margulis family and weak-* convergence. The uniqueness of disintegration implies that conditional measures of $\nu$ are also Margulis measures.

\end{proof}

 So, again using Proposition \ref{prop-Ledrappier} we obtain $h_{\mu}(\mathcal{F}^u, f) = \log(D^u).$

We may give another counting argument which is interesting in itself.

 For $\epsilon > 0$ let $E$ be a minimal $(n, \epsilon)$, $u-$spanning set of $\overline{\mathcal{F}^u(x, \delta)}.$ So, in particular $f^n(\mathcal{F}^u(x, \delta)) \subset \bigcup_{y \in f^n(E)} \mathcal{F}^u(y, \epsilon)$. By continuity  (or just uniform local finiteness of $m^u$)  we conclude that there exists $C_{\epsilon}> 0$ such that $C_{\epsilon}^{-1} \leq m^u(\mathcal{F}^u(y, \epsilon)) \leq C_{\epsilon}$ for any $y \in M.$ So, 
$$
 m^u(f^n(\mathcal{F}^u(x, \delta))) \leq C_{\epsilon} S^u(f, n, \epsilon, \delta).
$$
As $\log(m^u(f^n(\mathcal{F}^u(x, \delta))) = n \log(D^u) m^u(\mathcal{F}^u(x, \delta))$, we get
$$
h^u_{top}(f, \mathcal{F}^u(x, \delta)) = \lim_{\epsilon \rightarrow 0} \limsup_{n \rightarrow \infty} \frac{1}{n} \log S^u(f, n, \epsilon, \delta)  
\geq \log(D^u)
$$
Now, given any $\epsilon >0$ by definition of $h^u_{top}(f),$ there exists a point $x$ and $\epsilon_0 < \epsilon$ such that

$$
\limsup _{n \rightarrow \infty} \frac{1}{n}\log  N^{u}\left(f, \epsilon_{0}, n, x, \delta\right)>h_{top}^{u}(f)-\epsilon
$$Let $F \subset \overline{\mathcal{F}^u(x, \delta)}$ be an $(n, \epsilon_0)$ separated set. Here it is used that $\mathcal{F}^u$ is an expanding foliation. Since $f^n(F)$ is $\epsilon_0$ separated, the $\frac{\epsilon_0}{2}$  balls (in $u-$intrinsic distance) around points in $f^n(F)$ are disjoint subsets of $f^n(\mathcal{F}^u(x, \delta+ \epsilon_0/2)))$ and we get 
$$
 m^u (f^n(
 \mathcal{F}^u(x, \delta + \epsilon_0/2)) \geq C_{\epsilon_0/2} N^{u}(f, \epsilon_{0}, n, x, \delta).
$$
Again taking logarithm and dividing by $n$ we get
$$
 \log(D^u) \geq \limsup_{n \rightarrow \infty}\frac{1}{n} \log(N^{u}\left(f, \epsilon_{0}, n, x, \delta\right)) > h^u_{top}(f) - \epsilon. 
$$
Since $\epsilon$ is arbitrary we conclude $\log(D^u) \geq h^u_{top}(f).$ 

\end{proof}

As there always exists a measure which maximizes $u-$entropy (by upper semi continuity argument) one may ask the following question:

\begin{question}
Does any partially hyperbolic diffeomorphism admit a Margulis family of measures with dilation $\exp(h^u_{top}(f))?$
\end{question}

\begin{question}
Let $f$ be partially hyperbolic and $\mu$ an  ergodic invariant probability such that 
$ h_{\mu} (\cF^u) = h_{\mu}(f) = h_{top}(f).
$ Is it true that $\{\mu^u_x\}_{x \in \supp(\mu)}$ is a continuous family of Margulis measure?
\end{question}

Even in the case of  Anosov diffeomorphisms, there may be some interesting facts to be explored: 
Let $f: \mathbb{T}^2 \rightarrow \mathbb{T}^2$ Anosov and $\mu$ an invariant probability such that conditionals of $\mu$  along unstable foliation are defined for all $x \in \mathbb{T}^2$ and  holonomy invariant, then $\mu$ is the measure of maximal entropy\footnote{We thank an anonymous referee for observing this result}. To see this, observe that after conjugating $f$ to a linear Anosov diffeomorphism $A$ with conjugacy $h$, one conclude that $h_* \mu$ is an $A-$invariant measure which is invariant under translation along stable foliation. 
Similarly as in  \ref{lebesgue}, one conclude that $h_{*}\mu$ is Lebesgue measure which implies that $\mu$ is maximal entropy measure of $f$.

However, it is still interesting to understand the category of invariant measures whose disintegration along unstable foliation is a continuous family. Equilibrium states of H\"{o}lder potentials are among such measures.

\section{Rigidity, measure Rigidity and entropy} \label{s.rigidity}
In this section we discuss briefly some results where the notion of entropy along an expanding foliation is usefull to prove  rigidity results. 

Recall that for an invariant expanding foliation $\mathcal{F}$ we have defined expanding Gibbs state along $\mathcal{F}$ in Definition \ref{e-gibbs}. In the following theorem, the diffeomorphisms under consideration are sufficiently regular and such Gibbs states disintegrate absolutely continuous with respect to Lebesgue measure along the leaves of $\mathcal{F}.$
\subsection{Entropy along expanding foliations and Conjugacies} \label{rigidity-gibbs}

Now we announce a theorem of Saghin-Yang which is a technical general result which is useful to prove high regularity of conjugacies between dynamics. 

\begin{theorem} \cite{SaY} \label{SaYtheorem2}
Let $f, g \in \Diff^k(M), k \geq 2$, $\mathcal{F}^{f}$ and $\mathcal{F}^{g}$ be expanding foliations of $f$ and $g$ with uniform $C^r, r >1$ leaves and $\mu \in Gibb^{e}(f, \mathcal{F}^f)$. Suppose that $f$ and $g$ are conjugated by a homeomorphism $h$, and $h$ maps the foliation $\mathcal{F}^{f}$ to $\mathcal{F}^{g}$. Then the following conditions are equivalent:
\begin{enumerate}
\item $\nu:= h_* \mu$ is a Gibbs e-state for the foliation $\mathcal{F}^;.$
\item $h| \mathcal{F}^f$ is absolutely continuous on the supprt of $\mu$ (with respect to Lebesgue on $\mathcal{F}^f$ and $\mathcal{F}^g$), with continuous Jacobian  on $\supp(\mu)$, and bounded away from zero and infinity;
\item there exists $K > 0$, such that for every $x \in \supp(\mu)$ and any integer $n > 0$, 
$$
   \frac{1}{K} < \frac{\det (Df^n| T_x\mathcal{F}^f (x))}{\det (Dg^n| T_{h(x)}\mathcal{F}^g (h(x)))} < K;
$$
\item the sum of Lyapunov exponents along the expanding foliations are the same: 
$$
\lambda^{\mathcal{F}^f} (f, \mu):= \int \log (\det (Df|T\mathcal{F}^{f})) d\mu =  \int \log (\det (Dg|T\mathcal{F}^{g})) d\nu;
$$

\end{enumerate}
If the foliations $\mathcal{F}^f$ and $\mathcal{F}^g$ are one dimensional, then the above conditions are also equivalent to 
\begin{itemize}
\item $(5)^{'}$ $h$ restricted on each $\mathcal{F}^f$ leaf within the support of $\mu$ is uniformly $C^r$ smooth.
\end{itemize}  
\end{theorem}

\begin{proof}

In the above theorem, the fact that (4) implies (1) is proved  using the entropy along expanding foliations in a clear way. Here we write in more details this part of the proof.
For a complete proof reader should look at Theorem G in \cite{SaY}. 

By definition if $\mu$ is a Gibbs e-state along $\mathcal{F}^f$, then its conditional measures along $\mathcal{F}^f$ are absolutely continuous with respect to Lebesgue. However a similar to Ledrappier argument shows that in fact for each foliated chart $B,$ for any $\mu|B$ (normalized restriction of $\mu$ on $B$) almost every $x \in B,$ the disintegration of $\mu$ along $\mathcal{F}^f$ restricted to $B(x),$ equals $\rho(z) d vol_x$, where $\rho$ is continuous on $B$ and uniformly $C^{r-1}$ along the plaques of $\mathcal{F}^f| B.$ By $d vol_x$ we mean the normalized restriction of volume on the plaque of $\mathcal{F}^f$ passing through $x$ inside $B(x).$ 

The above fact about density of $\mu$ and $\nu$ along the corresponding expanding foliations is a fundamental step to  prove that (1) implies (2). 
The facts that (2) implies (3) and (3) implies (4) are relatively classical arguments. 
Let us prove that  (4) implies (1):

Let $\xi^f$ be a measurable partition subordinated to $\mathcal{F}^f$. As $h_* \mu = \nu$ and $h(\mathcal{F}^f) = \mathcal{F}^g$ we obtain that $h(\xi^f)$ is a measurable partition subsordinated  to $\mathcal{F}^g.$ As $h$ is a conjugact, $$ h_{\mu}(f, \mathcal{F}^f) = h_{\nu} (g, \mathcal{F}^g).  $$

As $\mu \in Gibb^e(f, \mathcal{F}^f),$ by Theorem \ref{SaYtheorem} we get $h_{\mu} (f, \mathcal{F}^f) = \lambda^{\mathcal{F}^f} (f, \mu).$ By hypothesis (4) we have
$\lambda^{\mathcal{F}^f} (f, \mu)= \lambda^{\mathcal{F}^g} (g, \nu)$  and consequently:
$$
 h_{\nu}(g, \mathcal{F}^g) = \lambda^{\mathcal{F}^g} (g, \nu).
$$
By theorem \ref{SaYtheorem}, $\nu \in Gibb^e(g , \mathcal{F}^g).$  

 Item $(5)^{'}$ for one dimensional foliations is the key result to prove rigidity of Anosov diffeomorphisms in $\mathbb{T}^3.$ (Theorems \ref{SaY}, \ref{MiTa}) It is a classical argument (See \cite{Llave}.) in smooth ergodic theory that if $h$ is a one dimensional homeomorphism between the $C^r-$curves $I_1$ and $I_2$, and there exist probability measures $\mu_i$ on $I_i$, absolutely continuous with respect to the Lebesgue measure on $I_i$, with $C^{r-1}-$densities bounded away from zero and infinity and such that $h_*\mu_1 = \mu_2,$ then $h$ is a $C^r$ diffeomorphism between $I_1$ and $I_2.$ 
\end{proof}

\subsubsection{Rigidity of Anosov diffeomorphisms} \label{rigidityanosov}

R. Saghin and J. Yang \cite{SaY} proved the following rigidity result showing that just Lyapunov exponents data is sufficient for smooth conjugacy.

\begin{theorem} \cite{SaY} \label{SaY}
Let $A$ be a three dimensional Anosov automorphism on $\mathbb{T}^3$  with simple real eigenvalues with distinct absolute values. Suppose $f$ is $C^{\infty}$ volume preserving Anosov diffeomorphism which is isotopic to $A$ and is partially hyperbolic. If $f$ has the same Lyapunov exponents as $A$, then $f$ is $C^{\infty}$ conjugate to $A.$
\end{theorem}
  Micena-Tahzibi \cite{MiTa} also proved a slightly differently announced result (it implies the above theorem) and the ideas for the proof are similar.
  \begin{theorem} \cite{MiTa} \label{MiTa} Let $f$  be  a  $C^r, r \geq 2,$  volume preserving Anosov diffeomorphism such that $T\mathbb{T}^3 = E^s_f \oplus E^{wu}_f \oplus E^{su}_f.$ (stable, weak unstable and strong unstable bundles) Then $f$ is $C^1$ conjugated with its linear part $A$ if and only if the center foliation $\mathcal{F}^{wu}_f$ is absolutely continuous and the equality $\lambda^{wu}_f(x) = \lambda^{wu}_A,$ between center Lyapunov exponents of $f$ and $A,$ holds for $m$ a.e. $x \in \mathbb{T}^3.$
\end{theorem}

Let us give a sketch of the proof of Theorem \ref{SaY} using Theorem \ref{SaYtheorem2}: Let $h_f$ be topological conjugacy between $f$ and $A.$ 
$f$ admits three invariant foliations $\mathcal{F}^s_f, \mathcal{F}^{wu}_f, \mathcal{F}^{su}_f$ which are tangent to $E^s$ (stable), $E^{wu}$ (weak unstable) and $E^{su}$ (strong unstable). It is known that 
$$
 h_{f}(\mathcal{F}^i_f) = \mathcal{F}^i_A, \quad \text{for} \quad i=s, wu.
$$
The matching between strong unstable foliation of $f$ and $A$ is a subtle problem. However, by a result due to Gogolev \cite{Gogu}, if $h_f$ is uniformly smooth along $\mathcal{F}^{wu}_f$ then $$h_f (\mathcal{F}^{su}_f)= \mathcal{F}^{su}_A.$$

Now, we see how Theorem \ref{SaYtheorem2} is used in the proof of this rigidity result. Indeed, as the stable Lyapunov exponents of $f$ and $A$ coincide and stable foliation is absolutely continuous (so volume has absolutely continuous disintegration along stable foliation), applying Theorem \ref{SaYtheorem2} for $f^{-1}$ and $A^{-1}$ we conclude that $(h_f)_* vol$ has absolutely continuous disintegration along $\mathcal{F}^s_A$ which implies that $(h_f)_* vol= vol$ and $h_f$ is uniformly $C^r$ along $\mathcal{F}^s_f.$ To obtain $C^r$ regularity we use one dimensionality of the stable foliation.

The next step is to prove that $h_f^{-1}$ along $\mathcal{F}^{wu}_A$ is uniformly $C^{1+ \epsilon}, \epsilon > 0.$ Indeed $\mathcal{F}^{wu}_f$ has uniformly $C^{1+\epsilon} $ leaves (by H\"{o}lder continuity of weak unstable bundle). Observe that volume is Gibbs e-state along the foliation $\mathcal{F}^{wu}_A$ , $h_f^{-1}(\mathcal{F}^{wu}_A)= \mathcal{F}^{wu}_f$ and $(h_f)^{-1}_* vol= vol$. So again theorem \ref{SaYtheorem2} implies that $h_f^{-1}$ is uniformly $C^{1+\epsilon}$ along $\mathcal{F}^{wu}_A.$

Finally as the (strong)unstable exponents of $f$ and $A$ coincide, using theorem \ref{SaYtheorem2} we conclude the $C^r$ regularity $h_f^{-1}$ along $\mathcal{F}^{su}_A.$

The last trick is to use Journ\'{e} result \cite{journe}  to conclude that $h_f^{-1}$ (and consequently $h_f$) is uniformly $C^{1+\epsilon}$. Then, by the bootstrapping argument of Gogolev \cite{gogboot}, the conjugacy $h_f$ is $C^{\infty}.$

\subsection{Measure Rigidity} \label{measurerigidity}
A celebrated conjecture of Furstenberg states that the only ergodic, Borel probability measure on $\mathbb{S}^1$ that is invariant under both 
$$ x \rightarrow 2x (mod-1) , \quad x \rightarrow 3x (mod-1)$$
is either supported on a finite set (of rational numbers) or is the Lebesgue measure on $\mathbb{S}^1.$


The conjecture remains open but Rudolph obtained an optimal partial result.

\begin{theorem} \cite{Rud}
The only ergodic Borel probability measure $\mu$ on $\mathbb{S}^1$ that is invariant under both $E_2: x \rightarrow 2x $(mod-1) and $E_3: x \rightarrow 3x$(mod-1) and satisfies 
$$ h_{\mu} (E_2) > 0 \quad \text{or} \quad h_{\mu} (E_3) >0$$
is the Lebesgue measure on $\mathbb{S}^1.$
\end{theorem}

In invertible setting, 
A. Katok proposed studying a related action on a familiar spaces: the action of two commuting hyperbolic automorphism of $\mathbb{T}^3.$

Katok-Spatzier (\cite{KS1}, \cite{KS2}) obtained a generalization of Rudolph theorem for these actions.  We follow the following example which contains many ideas of their result:

Let $A = \begin{bmatrix}
3& 2 & 1\\
2 & 2 & 1\\
1 & 1 & 1
\end{bmatrix}$ and $B = \begin{bmatrix}
2& 1 & 1\\
1 & 2 & 0\\
1 & 0 & 1
\end{bmatrix}$ and $L_A$ and $L_B$ the corresponding $\mathbb{T}^3$ hyperbolic automorphisms. Observe that $\det(A)=\det(B) =1$ and 
\begin{itemize}
\item $A$ has 3 distinct eigenvalues, $\chi_A^1 > 1 > \chi_A^2 > \chi_A^3 >0$
\item $B$ has 3 distinct eigenvalues, $\chi_B^1 > \chi_B^2 > 1 > \chi_B^3 > 0$
\item $AB= BA$
\item $A^k B^l = Id$ if and only if $k=l=0.$
\end{itemize}
By the above comments $L_A$ and $L_B$ generate a genuine action of $\mathbb{Z}^2$ on $\mathbb{T}^3.$

\begin{theorem} \label{rigidAB}
Let $L_A, L_B: \mathbb{T}^3 \rightarrow \mathbb{T}^3$ be as above. Then , the only ergodic, Borel probability measure $\mu$ on $\mathbb{T}^3$ that is invariant under both $L_A$ and $L_B$ and satisfies 
$$ h_{\mu} (L_A) > 0  \quad \text{or} \quad h_{\mu} (L_B) >0$$
is the Lebesgue measure on $\mathbb{T}^3.$ 
\end{theorem}

Here we will not give a proof of the above theorem. Instead we will just outline the role of unstable entropy in the proof. For a proof see \cite{Aaron}.

As $A$ and $B$ commute and are diagonalizable, we denote by $E^j$ the $j$th joint eigenspace of $A$ and $B.$ Let us denote by $\lambda_A^j = \log(\chi_A^j)$ and $\lambda_B^j = \log(\chi_B^j).$ 
For each $(n_1, n_2) \in \mathbb{Z}^2 ,$ the subspace $E^j$ is an eigenspace for $A^{n_1}B^{n_2}$  with eigenvalue $\chi^j(n_1, n_2)$ where
$$
 \log(\chi^j(n_1, n_2)) = n_1 \lambda_A^j + n_2 \lambda_B^j.
$$
In particular $\lambda^j : \mathbb{Z}^2 \rightarrow \mathbb{R}$ extends to a linear functional $\lambda^j : \mathbb{R}^2 \rightarrow \mathbb{R}$ which is called $j$th Lyapunov exponent functional for the action induced by $L_A$ and $L_B$.

The following lemma is basic and has crucial role in the proof of theorem. Its proof comes from the fact that all $E^j$ are totally irrational and the flow induced by them on $\mathbb{T}^3$ is uniquely ergodic.

\begin{lemma} \label{lebesgue}
A Borel probability measure $\mu$ on $\mathbb{T}^3$ is the Lebesgue (Haar) measure if and only if there exists $1 \leq j \leq 3$ such that the measure $\mu$ is invariant under the 1-parameter group of translations generated by $E^j.$
\end{lemma}
So to prove the Theorem \ref{rigidAB}, it is enough to verify that any ergodic measure satisfying the hypothesis of the theorem is invariant under translation for some $E^j.$

To this end, firstly we take $i \in \{1, 2, 3\}$ and $\overrightarrow{n}$ such that $\alpha(\overrightarrow{n})$ has one dimensional unstable bundle which coincides with $E^i.$
 Working with higher rank action (using suspension of action $\alpha$ to an $\mathbb{R}^2$ action) one proves that $\mu$ has absolutely continuous disintegration along unstable foliation of $\alpha(\overrightarrow{n}).$ This part of proof is long and we will not repeat it here.
 Then an argument like Pesin's entropy formula shows that in fact $h_{\mu} (\alpha(\overrightarrow{n})) = \lambda^i(\overrightarrow{n}).$

Now, we can use the following proposition to show that the conditional measures of $\mu$ along  unstable foliation are indeed the (one dimensional) Lebesgue measure.

\begin{proposition}
For any $i \in \{1,2,3\},$ fix $\overrightarrow{n} \in \mathbb{Z}^2$ such that 
\begin{itemize}
\item $\lambda^i (\overrightarrow{n}) > 0$
\item $\lambda^j(\overrightarrow{n}) < 0$ for $j \neq i.$
\end{itemize} then, if $h_{\mu}(\alpha(\overrightarrow{n})) = \lambda^i (\overrightarrow{n})$ then for $\mu-$a.e. $x,$ we have that $\mu_x^i = m^i_x$ where $\mu^i_x$ is the conditional measure of $\mu$ along the unstable foliation of $\alpha(\overrightarrow{n})$ and $m^i_x$ is the corresponding Lebesgue measure. Observe that the unstable foliation coincides with the orbit of the action generated by translation of $E^i.$
\end{proposition}

\begin{proof}The proof of the above proposition can be obtained along the lines of the proof of proposition \ref{prop-Ledrappier}. In fact the Lebesgue measure on the $E^j$ is a Margulis family with dilation $exp(\lambda^i(\overrightarrow{n})).$
\end{proof}

 So, we have proved that the conditional measures of $\mu$ along the orbits of action generated by translation of $E^i$ is Lebesgue and then it is not difficult to see that $\mu$ is invariant under these translations.
 Now Lemma \ref{lebesgue} ends the proof of Theorem \ref{rigidAB}.

\bibliography{ourbib-BFT2016}{}
\bibliographystyle{plain}
\end{document}